\documentclass[12pt, a4paper]{extarticle}
\usepackage{Paper_Style}

\usepackage{csquotes}
 \usepackage[
   backend=biber,
    sorting=nyt, 
    maxbibnames=99, 
  ]{biblatex} 
  
\title{Boundary vortices in the presence of a magnetic field }
\author[1]{Khalida Awashra}
\author[1]{Matthias Kurzke} 
    \affil[1]{School of Mathematical Sciences, University of Nottingham, University Park, Nottingham NG7 2RD, UK, \href{pmxka6@exmail.nottingham.ac.uk}{khalida.awashra@nottingham.ac.uk} \ \&\ \href{pmzmk@exmail.nottingham.ac.uk}{matthias.kurzke@nottingham.ac.uk}}
\begin{document}
\maketitle
\begin{center}
\textbf{July 20, 2025} 
\end{center}
\thispagestyle{fancy}

\section*{Abstract}
We study the behaviour of the magnetization vector field in a thin ferromagnetic film in the presence of an external in-plane constant magnetic field. We work on a specific thin-film regime where boundary vortices dominate the energy and show the Gamma-convergence at the second order of the micromagnetic energy to an energy term called the renormalized energy, representing the interaction energy between boundary vortices. We present a new approach to defining the renormalized energy and rigorously show the relation between this definition and the classical one. We prove the concentration of the energy around boundary vortices, where the location of these vortices depends on the external field applied. We provide some numerical simulations of the magnetization vector field and the renormalized energy versus the location of the vortices in two different domains: the unit disk and an oval-shaped domain. 

\tableofcontents
\section{Introduction}
The theory of micromagnetics focuses on studying the behaviour of a three-dimensional unit vector field, called the \textit{magnetization}, in ferromagnetic samples. Mainly, it studies the stable states of this unit vector field, which correspond to the local minimizers of the micromagnetic energy of the sample. This variational problem is non-local, non-convex and multiscale, which makes the problem rich and interesting to study. 

\subsection{Physical background and the three-dimensional model}
We start by introducing our cylindrical shape ferromagnetic sample
\begin{equation*}
    \bfom= \omega \times (0,t) \subset \R^3,
\end{equation*}
where  $ \omega \subset \R^2$ is the cross-section of $\bfom$ with diameter $\ell>0$, and $t>0$ is the thickness of the sample $\bfom$. We assume that  $ \omega \subset \R^2$ is an open, bounded, simply-connected set of class $C^{1,1}$.
Note that we use the bold font in this paper to denote three-dimensional quantities.

The magnetization unit-length vector field describes the behaviour of the magnetic moments in our ferromagnetic sample, and it is given by
\begin{equation*} 
    \bfm=(m,m_3): \bfom \rightarrow \Ss^2,
\end{equation*}
where $\Ss^2$ is the unit sphere in $\R^3$, $m\in\R^2$ and $m_3\in\R$. The constraint $|\bfm|=1$ makes the problem non-convex. The three-dimensional micromagnetic energy of the domain $\bfom$ is
\begin{equation}
    E^{3D}(\bfm) = d^2 \int_{\bfom} |\bfa \bfm|^2 \,d\bfx + \int_{\R^3}|\bfa U|^2 \,d\bfx  + Q\int_{\bfom} \phi(\bfm) \, d\bfx- \int_{\bfom }\bfm\cdot\boldsymbol{H}_{ext} \,d\bfx ,
\end{equation}
where $\bfx= (x,x_3)= (x_1,x_2,x_3) \in \R^3$, $d\bfx$ is the three-dimensional Lebesgue measure, and $\bfa = (\nabla,\partial_{x_3})= (\partial_{x_1}, \partial_{x_2},\partial_{x_3})$.  
 The first term of the energy $E^{3D}$ is the \textit{exchange energy}, which is generated by interactions of neighbouring electrons spins. This energy term is lower when electron spins are parallel, i.e. it prefers constant (or slowly changing) magnetization. The positive constant $d$ is called the \textit{exchange length} and it is a material constant of the order of nanometers.

 The second term is the \textit{stray-field} (or the \textit{magnetostatic}) energy. It is a non-local term and represents the energy of the magnetic field induced by the magnetization $\bfm$ in the whole space $\R^3$. The \textit{stray field potential} $U: \R^3 \rightarrow \R$ is the solution of the static Maxwell equation
\begin{equation}\label{staticMax}
\bfd U = \bfa\cdot\left(\bfm \mathds{1}_{\bfom}\right)
\end{equation}
in the sense of distributions in $\R^3$, i.e. for every $\zeta \in C^\infty_c(\R^3)$ 
\begin{equation}\label{distMax}
    \int_{\R^3} \bfa U\cdot \bfa\zeta d \bfx = \int_{\bfom} \bfm \cdot \bfa \zeta d\bfx, 
\end{equation}
where $\mathds{1}_{\bfom}$ above is the characteristic function of $\bfom$. This energy term prefers divergence-free magnetizations in the sense of distributions, so it also prefers $\bfm \cdot \boldsymbol{\nu} =0$ on $\partial\bfom$, where $\boldsymbol{\nu}$ is the outer unit normal vector at $\partial\bfom$.

The third term of the micromagnetic energy is the \textit{anisotropy energy}, which describes the tendency of the magnetization vector field $\bfm$ to align with specific directions depending on the crystalline structure of the ferromagnetic material. The function $\phi:\Ss^2 \rightarrow \R$ is called the \textit{anisotropy energy density}, whose minima are the preferred directions of the magnetization (easy axes). The constant $Q$ is the \textit{quality factor}, which is a non-dimensional material constant that measures the strength of the anisotropy energy relative to the other terms of the energy $E^{3D}$. 

The last term of $E^{3D}$ appears in the presence of an \textit{external magnetic field} $\boldsymbol{H}_{ext}:\R^3\rightarrow\R^3$, and it is called the \textit{Zeeman}  (or the \textit{external field}) energy. This energy term favours the magnetization to align with the direction of the external field. For more details about these energy terms, see for example \cite{nanomag}, \cite{hubert}, \cite{aharoni}, and \cite{recent}.

The variational problem we study depends on many parameters that appear in the expression of the micromagnetic energy. Two of these parameters depend on the material of the ferromagnetic sample (the quality factor $Q$ and the exchange length $d$) and the other two depend on the geometry of the sample (the diameter of the cross-section $\ell$ and the thickness $t$). From these parameters, we introduce two positive non-dimensional parameters:
\begin{equation}
    \alpha:= \frac{t}{\ell}, \ \ \text{and} \ \   \eta := \frac{d}{\ell} ,
\end{equation}
where $\alpha$ is the aspect ratio of $\bfom$ and $\eta$ is the normalized exchange length. We are interested in studying a thin-film limit of this problem, i.e. when the thickness $t$ of the film is very small compared with the diameter $\ell$ ($\alpha \ll 1$). We also assume that our ferromagnetic material is soft with $Q=0$ (no anisotropy term will appear in our energy), and the external field $\boldsymbol{H}_{ext}$ is a constant in-plane vector ($\boldsymbol{H}_{ext}=(H_{ext},0)=(H_1,H_2,0)$, where $H_1,H_2\in \R$), and $|\boldsymbol{H}_{ext}|= O(d^2)$. 
\subsection{A thin-film regime}
\subsubsection{Energy rescaling} We are interested in studying the micromagnetic energy and the magnetization behaviour in a specific thin-film regime. We start with a rescaling of the micromagnetic energy by nondimensionalizing our parameters in length, i.e. let 
\begin{equation}
    \hat{\bfx} = \frac{\bfx}{\ell} \in \bfOm_\alpha\,, \ \ \text{for} \ \ \bfx\in\bfom, \ \ \text{where} \ \ 
 \ \bfOm_\alpha = \Omega\times (0,\alpha) \subset \R^3 \,, \ \ \Omega=\frac{\omega}{\ell}.
\end{equation}
The cross-section $\Omega$ now has diameter one (by definition). Let us also rescale the magnetization and the stray field potential by considering $\bfma(\hat{\bfx})  = \bfm(\bfx)$ and $U_\alpha(\hat{\bfx}) = \frac{1}{\ell}U(\bfx)$.

Now, we rescale the three-dimensional micromagnetic energy and get the \textit{rescaled energy} $\hat{E}_\alpha(\bfma;\hext)= \frac{1}{d^2 t |\log \varepsilon|} E^{3D}(\bfm)$, where 
\begin{equation}
    \varepsilon =\frac{\eta^2}{\alpha|\log\alpha|}.
\end{equation}
After dropping the hat $\ \hat{} \ $ from our notation (for simplicity), this energy can be written as
\begin{equation}\label{res.energy}
    E_\alpha(\bfma;\hext) = \frac{1}{\alpha |\log \varepsilon|}\int_{\bfOm_\alpha}|\nabla \bfma|^2 \,d\bfx +\frac{1}{\eta^2 \alpha  |\log \varepsilon|} \int_{\R^3}|\nabla U_\alpha|^2 \,d\bfx  - \frac{1}{d^2 \alpha |\log \varepsilon|} \int_{\boldsymbol{\Omega}_\alpha }\boldsymbol{H}_{ext}\cdot\bfma \,d\bfx.
\end{equation}
 The material constant $\frac{1}{d^2}$ appearing in the last term above tells us how different materials react to the same external field, but we will rescale the external field so that the new external field is $\hext = \frac{\boldsymbol{H}_{ext}}{d^2}$. Therefore, the rescaled energy \eqref{res.energy} is now given by 
\begin{equation} \label{rescaled}
E_\alpha(\bfma;\hext) = \frac{1}{\alpha |\log \varepsilon|}\int_{\bfOm_\alpha}|\nabla \bfma|^2 \,d\bfx +\frac{1}{\eta^2 \alpha  |\log \varepsilon|} \int_{\R^3}|\nabla U_\alpha|^2 \,d\bfx  - \frac{1}{ \alpha |\log \varepsilon|} \int_{\bfOm_\alpha } \hext\cdot\bfma \,d\bfx , 
\end{equation}
where 
\begin{equation}
    \bfma: \bfOm_\alpha \rightarrow \mathbb{S}^2, \ \ \text{and} \ \ \bfd U_\alpha = \bfa \cdot (\bfma \mathds{1}_{\bfOm_\alpha}) \ \ \text{in the sense of distributions in} \ \ \R^3 .
\end{equation}

\subsubsection{The regime}
We mentioned that the problem we study is multiscale, which makes the problem rich and allows to study it in different regimes. Here we focus on the following thin-film regime:
\begin{equation}\label{regime}
    \eta \ , \ \alpha \ll 1 \  \ \ \  \text{and} \ \ \ \frac{1}{|\log \alpha|} \ll \varepsilon \ll 1,
\end{equation}
which is equivalent to $\alpha \ll \eta^2 \ll \alpha|\log\alpha| \ll 1 $. We also consider the following narrower regime, where some of our results hold only in this regime:
\begin{equation}\label{narrower.regime}
   \frac{\log|\log\alpha|}{|\log\alpha|} \ll \varepsilon.
\end{equation}
This thin-film regime \eqref{regime} and the narrower one \eqref{narrower.regime} were studied by Ignat and Kurzke \cite{ignatk}. Using the notion of global Jacobian and Gamma-convergence, they showed that the micromagnetic energy concentrates around boundary vortices. However, they studied the problem with no external field (and no anisotropy energy). What we do here is trying to generalize their results when an in-plane constant external field $\hext$ is applied to the sample and prove the concentration of the energy around boundary vortices. We will see examples of how an applied field affects the positions of the boundary vortices and the magnetization inside the sample via numerical simulations.
\subsection{Energy reduction}
\subsubsection{From the 3D energy to a 2D energy}
Based on the work of Kohn and Slastikov in \cite{kohn}, we define the average magnetization $\bm = (\bmp,\bmt) : \Omega \rightarrow\bar B^3$ (where $\bar B^3$ is the closed unit ball in $\R^3$) as follows:
\begin{equation}\label{average.magnetization}
   \bm(x) = \frac{1}{\alpha} \int_{0}^{\alpha} \bfm_\alpha(x,x_3)\, dx_3 \ , \ \ x \in \Omega \subset  \R^2.
\end{equation}
The stray-field potential $\bu : \R^3 \rightarrow \R$ 
 is generated by $\bm$ via the equation 
\begin{equation}
    \boldsymbol{\Delta} \bu = \bfa \cdot (\bm \mathds{1}_{\bfOm_{\alpha}}) \ \ \ \text{in} \ \ \R^3 .
\end{equation}
Note that $|\bm| \leq 1$, and defining the average magnetization $\bm$ as the average over the third coordinate here makes sense as $\alpha$ tends to zero in our regime. 

We want to reduce the full 3D nonlocal micromagnetic energy to a 2D local energy functional for the average magnetization $\bm$. Therefore, following the work done by Ignat and Kurzke \cite{ignatk}, we introduce the following two-dimensional \textit{reduced energy} functional corresponding to $\bm$: 
\begin{align} 
\eeeb(\bm;\hext) = \frac{1}{|\log \varepsilon|}\Big(&\int_{\Omega}|\nabla \bm|^2 \,dx +\frac{1}{\eta^2 } \int_{\Omega}( 1-|\bmp|^2 )\,dx  \nonumber\\
&+ \frac{1}{2 \pi \varepsilon} \int_{\partial\Omega}(\bmp \cdot \nu)^2 d\mathcal{H}^1- \int_{\Omega} \hext\cdot\bm \,dx \Big),\label{reduced}
\end{align}
where $\mathcal{H}^1$ is the one dimensional Hausdorff measure, $\nu$ is the outer unit normal vector at $\partial\Omega$ and $\Omega$ is the 2D simply connected $C^{1,1}$ domain introduced above. This energy functional can be extended to all functions in the Sobolev space $H^1(\Omega, \R^3)$ by taking $\eeeb(\bfm) = \infty$ when $|\bfm| >1$ on a set of positive measure.

 Again, the main idea of our work is to study and analyze the minimizers of the micromagnetic energy, which are the stable state magnetization vector fields. The full $3D$ micromagnetic energy $E_\alpha$ \eqref{rescaled} is nonlocal (due to the magnetostatic term), but the reduced 2D energy $\bar E_\alpha$ is local as we see from equation \eqref{reduced}, so it is easier to study this energy. Our main aim at this stage is to rewrite Theorem 9 in \cite{ignatk} with an external magnetic field. This theorem studies the rescaled energy \eqref{rescaled} 
 and it has some compactness results and two orders of energy expansion by $\Gamma$-convergence. To prove this theorem we need to rewrite and prove some other theorems from \cite{ignatk} and \cite{ignatk2} in the presence of $\hext$. 
 
Let us start by rewriting Theorem 1 in \cite{ignatk} in the presence of an external field, which connects the $3D$ micromagnetic energy $E_\alpha$ \eqref{rescaled} with the $2D$ reduced energy $\eeeb$ \eqref{reduced} for the average magnetizations $\bm$ that are defined on the $2D$ transversal section of our sample. In other words, the following theorem reduces the energy $E_\alpha$ to the energy $\eeeb$.

\begin{theorem}\label{rescaled+reduced}
 Let $\bfOm_\alpha=\Omega\times (0,\alpha)$ with $\Omega\subset \R^2$ be a simply connected 
$C^{1,1}$ domain. In the regime \eqref{regime}, we consider a family of magnetizations $\{\bfm_\alpha:\bfOm_\alpha\to \Ss^2\}_{\alpha\to 0}$ 
with associated  stray field potentials $\{U_\alpha:\R^3\to \R\}_{\alpha\to 0}$  and we assume
\begin{equation*}
    \limsup_{\alpha\to 0} E_\alpha(\bfma;\hext)<\infty.
\end{equation*}
Then
\begin{equation}\label{lowerbound11}  E_\alpha(\bfma;\hext)\ge \eeeb(\bm;\hext)-o(1) \quad \textrm{as} \quad \alpha\to 0.
    \end{equation}
Moreover, in the narrower regime \eqref{narrower.regime} , we have the following improved estimate:
\begin{equation}\label{narrower.inequality}
E_\alpha(\bfm_\alpha;\hext)\ge \eeeb(\bm;\hext)-o(\frac{1}{|\log\varepsilon|})\quad \textrm{as} \quad \alpha\to 0.\end{equation}
If $\bfm_\alpha$ are independent of $x_3$ (i.e. $\bfm_\alpha=\bm$), then 
in the regime \eqref{regime} there holds 
$E_\alpha(\bfm_\alpha;\hext)= \eeeb(\bm;\hext)-o(1)$, 
while in the regime \eqref{narrower.regime} we have $E_\alpha(\bfm_\alpha;\hext)= \eeeb(\bm;\hext)-o(\frac{1}{|\log\varepsilon|})$
as $\alpha\to 0$.
\end{theorem}
\subsubsection{Another energy reduction}
Let us introduce the following energy functional for two-dimensional maps $u\in H^1(\Omega;\R^2)$:
\begin{equation}\label{epsilon.eta.energy.definition}
   \eee(u;\hex) = \int_{\Omega}|\nabla u|^2 \,dx +\frac{1}{\eta^2 } \int_{\Omega}( 1-|u|^2 )^2\,dx  + \frac{1}{2 \pi \varepsilon} \int_{\partial\Omega}(u \cdot \nu)^2 d\mathcal{H}^1- \int_{\Omega} \hex\cdot u \,dx . 
\end{equation}
This energy is similar to the energy functional studied in \cite{ignatk2}, but here we have an additional term due to the presence of an external magnetic field. If we take $u$ to be the average magnetization vector, i.e. $u=\bmp$ (where $\bm=(\bmp,\bmpth)$), then comparing the second term of $\eee(\bmp;\hex)$ with the second term of $\eeeb(\bm;\hext)$, we get 
\begin{equation}\label{lowerbound2}
    \eeeb(\bm;\hext) \geq \frac{1}{|\log\varepsilon|} \eee(\bmp;\hex),
\end{equation}
(since $\left(1-|u|^2\right) \geq \left(1-|u|^2\right) ^2$, for $|u|\leq1$).  From equation \eqref{lowerbound11} and the above inequality \eqref{lowerbound2}, we see that finding lower bounds for the energy functional $\eee$ will provide lower bounds for $\eeeb$, which in turn provide lower bounds for $E_\alpha$. Following the work of Ignat and Kurzke  \cite{ignatk2}, we can simplify the problem more by introducing another energy functional using a lifting argument, and then we will state and prove some results regarding the compactness and $\Gamma$-convergence for this new energy. Let us introduce this new functional by rewriting lemma 4.1 in \cite{ignatk2} with $\hex$. 
\begin{lemma}\label{lifting}
    Let $\Omega \subset \R^{2}$ be a bounded, simply connected and $C^{1,1}$ domain. If $u \in$ $H^{1}\left(\Omega ; \Ss^1\right)$ then there exists a lifting $\phi \in H^{1}(\Omega ; \R)$ with $u=e^{i \phi}$ and $\phi$ is unique up to an additive constant in $2 \pi \mathbb{Z}$. Furthermore, for every small $\varepsilon>0$ and $\eta>0$,
\begin{equation}\label{Gepsi.definition}
  \Ge(\phi;\hex):=  E_{\varepsilon, \eta}(u;\hex) = \int_{\Omega}|\nabla \phi|^{2} d x+\frac{1}{2 \pi \varepsilon} \int_{\partial \Omega} \sin ^{2}(\phi-g) d \mathcal{H}^{1} - \int_{\Omega} \hex\cdot e^{i\phi} \,dx  , 
\end{equation}
where $g$ is a lifting of the unit tangent vector field $\tau$ at $\partial \Omega$, i.e.
\begin{equation}\label{liftingtangent}
    e^{i g}=\tau=i \nu \quad \text { on } \quad \partial \Omega.
\end{equation}
\end{lemma}
The proof of this lemma is similar to the proof of lemma 4.1 in \cite{ignatk2}.

\subsection{The renormalized energy}
The \textit{renormalized energy} represents the interaction energy between boundary vortices in the sample. As we will see later, this energy appears asymptotically in the second-order expansion of the micromagnetic energy in the sense of Gamma-convergence.  
%
%
%

In a similar context to that introduced by Brezis-Bethuel-Helein \cite{Brezis}, we firstly state the definition of the renormalized energy when no external field is applied ($\hex=0$) (see Definition 5 in \cite{ignatk}). 
\begin{definition}\label{unperturbed.renormalized.def}
 Let $\Omega \subset \mathbb{R}^2$ be a bounded, simply connected, $C^{1,1}$ domain and $\varkappa$ the curvature of $\partial\Omega$. Let $\phib:\partial\Omega \rightarrow \mathbb{R}$ be a BV function such that $e^{i\phib}\cdot\nu = 0$ on $\partial\Omega \backslash \{ a_1, ..., a_N\}$ for $N\ge 2$ distinct points $a_j \in \partial\Omega$ and 
\begin{equation*}
    \partial_{\tau}\phib = \varkappa - \pi \sum_{j=1}^{N} d_j \delta_{a_j}\ \  \text{on} \ \partial\Omega\ \ \text{with} \  d_j\in\{ \pm1\} \ \ \text{and} \ \sum_{j=1}^{N} d_j = 2 . 
\end{equation*}
If $\phi_*$ is the harmonic extension to $\Omega$ of $\phib$, i.e. if it satisfies the following equation 
 \begin{equation}
    \begin{cases} \Delta\phi_*  =0 \ \ &\text{in}  \ \Omega , \\
    \phi_*=\phib \ \  &\text{on} \ \partial\Omega .
    \end{cases}
     \end{equation} 
 Then the \textit{renormalized energy} of $\{(a_j, d_j)\}$ in the absence of an external field is
 \begin{equation}
      W_{\Omega;0}(\{ (a_j,d_j)\}) = \lim_{\rho \to 0}\left(  \int_{\Omega\backslash \cup_{j=1}^{N} B_\rho(a_j)}|\nabla \phi_*|^2 \,dx   - \pi N \log \frac{1}{\rho}\right),
 \end{equation}  
 where $ B_\rho(a_j)$ is the disk centered at $a_j$ of radius $\rho$. 
\end{definition}

As we see from the definition above, the renormalized energy eliminates the 'infinite' energy concentrated in the small disks $ B_\rho(a_j)$ around the boundary vortices $a_j \in \partial\Omega$, for $j=1, 2, ...,N$.

When an external field $\hex$ is applied to the domain, we define the renormalized energy as follows.

\begin{definition}\label{renormalized.definition}
  Let $\Omega \subset \mathbb{R}^2$ be a bounded, simply connected and $C^{1,1}$ domain. Let  $\sigmas$ be any minimizer of 
\begin{equation*}
    F(\sigma) = \int_\Omega \left(|\nabla \sigma|^2   - h_{ext}\cdot e^{i(\sigma+\phi_*)}\right) \, dx,
\end{equation*}
where $\phi_*$ is the function mentioned in Definition \ref{unperturbed.renormalized.def} and $\sigmas=0$ on $\partial\Omega$. Then the renormalized energy in the presence of an external field is 
\begin{equation}\label{renormalized.energy}
    \Wad = W_{\Omega;0}(\{ (a_j,d_j)\}) + F(\sigmas),
\end{equation}
where $ W_{\Omega;0}(\{ (a_j,d_j)\})$ is the unperturbed renormalized energy defined above.
\end{definition}
\begin{remark}
    The functional $F(\sigma)$ has at least one minimizer in $H^1(\Omega;\R)$. This follows from the Direct Method of the calculus of variations (see for example Theorem 2.3 in \cite{rindler}). Moreover, for sufficiently small magnetic field $\hex$, the functional $F$ is convex on $H^1_0(\Omega,\R)$. Hence, there will be a unique minimizer of $F$ in $H^1_0(\Omega,\R)$, which is the solution of the Euler-Lagrange equation associated with $F$.  
\end{remark}
This new definition of the renormalized energy is a better approach to defining the interaction energy between boundary vortices as it does not require the uniqueness of the solution of the Euler-Lagrange equation associated with the renormalized energy. It is kind of similar to the definition in \cite{appliedfield}, but here we show rigorously, in the next lemma, that this new renormalized energy is lower than any energy of the form of $G$ defined below in \eqref{Genergy}. The energy $G$ is inspired by the previous definition of the unperturbed renormalized energy (Definition \ref{unperturbed.renormalized.def}), but here we have an additional term due to the existence of an external magnetic field. 
\begin{lemma}\label{renormalized_lemma}
    Let $\psi\in H^{1}_{loc}(\bar{\Omega}\setminus \{a_1,...,a_N\})$ with $\psi=\phib$ on the boundary of $\Omega$ and let 
    \begin{equation}\label{Genergy}
        G(\psi) := \liminf_{\rho \to 0}\left(  \int_{\Omega\backslash \cup_{j=1}^{N} B_\rho(a_j)}|\nabla \psi|^2 \,dx   - \pi N \log \frac{1}{\rho}\right) - \int_{\Omega} h_{ext}\cdot e^{i\psi} \, dx.
    \end{equation}
    Then, $G(\psi)\geq \Wad$, with equality when $\psi=\phi_* + \sigmas$, where $\phi_*$ and $\sigmas$ are the functions mentioned above.
\end{lemma}
We will prove this lemma in Section 3.
\subsection{Global Jacobian}
Now we introduce the global Jacobian, which helps detect the topological singularities in both the interior and the boundary of the domain (i.e. in detecting interior and boundary vortices in the sample). Here we mention the definition of the global Jacobian, and we refer to section 1.1 in \cite{ignatk2} for more details. 

\begin{definition}
    Let $u\in H^1(\Omega;\R^2)$, where $\Omega \subset \R^2$ is a Lipschitz bounded domain. We define the \textit{global Jacobian} of $u$ as the linear functional $\mathcal{J}(u): W^{1,\infty}(\Omega) \rightarrow \R$ acting on Lipschitz test functions as follows:
    \begin{equation*}
        \langle \mathcal{J}(u), \zeta \rangle := - \int_\Omega \left( u\times\nabla u \right) \cdot\nabla^{\perp} \zeta \, dx, \ \ \ \text{for every Lipschitz function} \ \ \zeta:\Omega\rightarrow\R,
    \end{equation*}
    where $a\times b =a_1b_2-a_2b_1$ for $a=(a_1,a_2) \in \R^2$ and $b=(b_1,b_2) \in \R^2$, and $\nabla^\perp=(-\partial_{x_2}, \partial_{x_1})$. 
\end{definition}
We also define the \textit{interior Jacobian} of $u\in H^1(\Omega,\R^2)$ as 
\begin{equation*}
    \text{jac} (u) = \text{det} (\nabla u)= \partial_{x_1} u \times \partial_{x_2} u \in L^1(\Omega), 
\end{equation*}
and the \textit{boundary Jacobian} $\mathcal{J}_{bd}:W^{1,\infty}(\Omega) \rightarrow \R$  as
\begin{equation*}
    \langle \mathcal{J}_{bd}(u), \zeta \rangle := \langle \mathcal{J}(u), \zeta \rangle - 2\int_\Omega \text{jac}(u) \zeta \, dx,   \ \ \ \text{for every Lipschitz function} \ \ \zeta:\Omega\rightarrow\R.
\end{equation*}
The boundary Jacobian functional $ \mathcal{J}_{bd}$ acts only on the boundary of the domain, i.e. on $\partial \Omega$ (see proposition 2.2 in \cite{ignatk2}). Therefore, if $u\in C^2(\bar{\Omega}; \R^2)$, then integration by parts implies 
\begin{equation*}
  \langle \mathcal{J}_{bd}(u), \zeta \rangle = - \int_{\partial\Omega}\left( u\times \partial_\tau u \right) \zeta \, d\mathcal{H}^1,   \ \ \ \text{for every Lipschitz function} \ \ \zeta:\Omega\rightarrow\R ,   
\end{equation*}
where $\tau$ is the unit tangent vector at $\partial\Omega$. Note that $\zeta$ was defined only on $\Omega$, but a Lipschitz function $\zeta:\Omega\rightarrow\R$ has a unique Lipschitz extension to $\bar{\Omega}$ that we consider here. 

If $u:\Omega\rightarrow \mathbb{S}^1$ has a smooth lifting $\varphi\in C^2(\bar{\Omega}; \R)$, i.e. $u=\left(\cos\varphi, \sin\varphi\right)$ in $\Omega$, then the interior Jacobian vanishes in $\Omega$. This means that the tangential derivative of the smooth lifting $\varphi$ at the boundary carries the whole topological information of the sample and we get
\begin{equation*}
 \text{jac}(u)=0, \ \ \mathcal{J}(u) = \mathcal{J}_{bd}(u) = -\partial_\tau\varphi \mathcal{H}^1   \llcorner \partial\Omega \ \ \text{and} \ \  \langle \mathcal{J}_{bd}(u), 1 \rangle = 0  \ \ \text{if} \ \ u=e^{i\varphi} \ \ \text{in} \ \ \Omega.
\end{equation*}

\subsection{Main results}
One of our main aims is to state and prove some compactness and $\Gamma$-convergence results for the energy functional $E_\alpha$. We summarize these results in Theorem \ref{rescaled.energy.compact}. This theorem is similar to Theorem 9 in \cite{ignatk}, but here the external field effect appears which makes the theorem a generalization of that in \cite{ignatk}. In this theorem, under a specific energy bound in our regime \eqref{regime}, we first show a compactness result of the global Jacobian. In other words, we prove that the global Jacobian of the average magnetizations converges to a measure supported at the boundary of the sample $\partial\Omega$, and this measure has a finite sum of Dirac measures $\delta_{a_j}$ at the boundary vortices $a_j\in \partial\Omega$. Each of these boundary vortices $a_j$ has a nonzero multiplicity $d_j\in \mathbb{Z}\setminus \{0\}$ and the sum of their multiplicities equals two. Then we show a Gamma-convergence result at the first order for the rescaled energy $E_\alpha$ and prove that the Gamma-limit depends on the multiplicities of the boundary vortices detected by the global Jacobian. Finally, in the narrower regime \eqref{narrower.regime} and under a second-order energy bound, we show the convergence at the second order of the rescaled energy $E_\alpha$ to a Gamma-limit containing the renormalized energy $\W$ defined above, which depends on the positions of the boundary vortices. 
\begin{theorem}\label{rescaled.energy.compact}
     Let $\Omega\subset \R^2$ be a bounded, simply connected, $C^{1,1}$ domain. 
If $\alpha\to 0$, $\eta=\eta(\alpha)\to 0$ and $\varepsilon=\varepsilon(\alpha)\to 0$ satisfy the regime \eqref{regime}, then the following holds: 
Assume  $\bfm_\alpha\in H^1(\bfOm_\alpha; \Ss^2)$ is a sequence of magnetizations such that
$$\limsup_{\alpha\to 0} E_\alpha(\bfm_\alpha;\hext)<\infty$$ with $E_\alpha$ defined in \eqref{rescaled} and let $\bm=(\bar m_\alpha, \bar m_{\alpha,3})$ be the average magnetization vector field defined in \eqref{average.magnetization}. 
\begin{enumerate}
\item \textbf{Compactness of the global Jacobian and of the traces $\bm\big|_{\partial \Omega}$.} For a subsequence, the global Jacobians of the in-plane average magnetizations $\mathcal{J}(\bar m_\alpha)$ converge  
 to a measure $J$ on the closure $\bar \Omega$ 
in the following sense:
 \begin{equation*}
  \label{conv_lip}
  \sup_{|\nabla \zeta|\leq 1\textrm{ in } \Omega}\left|\left<\mathcal{J}(\bar m_\alpha)-J,\zeta\right>\right|\to 0 \quad \textrm{as }\, \alpha\to 0,
  \end{equation*}
where $J$ is supported on $\partial\Omega$ and has the form 
\begin{equation*}
\label{newlab}
J=-\varkappa \mathcal{H}^1{\llcorner \partial\Omega}+\pi \sum_{j=1}^N d_j \delta_{a_j} \quad \textrm{with} \quad  \sum_{j=1}^N d_j=2
\end{equation*}
for $N\geq 1$ distinct boundary vortices $a_j \in \partial\Omega$ carrying the non-zero multiplicities
$d_j\in\mathbb{Z}\setminus \{0\}$. 
Moreover, for a subsequence, the trace of the average magnetizations $\bm\big|_{\partial \Omega}$ converges as $\alpha\to 0$ in $L^p(\partial\Omega)$ ({for every $p\in [1, \infty)$}) to $(e^{i\phi},0)\in  BV(\partial\Omega; \Ss^1\times\{0\})$  
for a $BV$ lifting $\phi$ of the tangent field $\pm \tau$ on $\partial\Omega$ determined (up to a constant in $\pi \mathbb{Z}$) by $\partial_\tau \phi =-J$ on $\partial\Omega$.

\item \textbf{First order lower bound.} The energy satisfies
\begin{equation}
\liminf_{\alpha\to 0} E_\alpha(\bfm_\alpha;\hext) \ge \pi\sum_{j=1}^N |d_j|.
\end{equation}
\item \textbf{Single multiplicity and second order lower bound.} If additionally {$\frac{\log|\log \alpha|}{|\log \alpha|}\ll \varepsilon$} and
\begin{equation}
\limsup_{\alpha\to 0} |\log \varepsilon| \left(E_\alpha(\bfm_\alpha;\hext) - \pi\sum_{j=1}^N |d_j|\right)<\infty, 
\end{equation}
then 
the multiplicities of the vortices $a_j$ satisfy $d_j=\pm 1$ for $1\leq j\leq N$ and
 the finer
energy lower bound holds:
\begin{equation}
\liminf_{\alpha\to 0} |\log \varepsilon| \left(E_\alpha(\bfm_\alpha;\hext) - \pi N \right) 
\ge \Wad + {\gamma_0}N, 
\end{equation}
where $\gamma_0=\pi\log\frac{ e}{4\pi}$ is a constant and $ \W$ is the renormalized
energy defined in \eqref{renormalized.energy}.

\item \textbf{Strong compactness of the {rescaled} magnetization.} Under the assumptions in part 3, we also have for every $q\in[1,2)$ the bound
\[
\limsup_{\alpha\to 0}\frac1\alpha\int_{\bfOm_\alpha} |\bfa \bfm_\alpha|^q \, dx <\infty.
\]
For a subsequence we have that  $\tilde \bfm_\alpha(x,x_3):\bfOm_\alpha\to \Ss^2$ defined by 
 $\tilde \bfm_\alpha(x,x_3)=\bfm_\alpha(x,\alpha x_3)$ converges strongly in every $L^p(\bfOm_\alpha)$, $1\le p<\infty$,  to a $W^{1,q}$-map  $\tilde\bfm=(\tilde m, 0)$ with
$|\tilde\bfm|=|\tilde m|=1$ and $\partial_{x_3}\tilde\bfm=0$, i.e. $\tilde \bfm=\tilde \bfm(x)\in W^{1,q}(\Omega, \Ss^1\times \{0\})$ for every $q\in[1,2)$. Moreover, the global Jacobian $\mathcal{J}(\tilde m)$ coincides with the measure $J$ on $\bar \Omega$ given in part 1 above.
\end{enumerate}
\end{theorem}
To prove this theorem, we use the following:
\begin{itemize}
    \item Theorem 9 in \cite{ignatk}, which has compactness and Gamma-convergence results for the energy $E_\alpha$ in the absence of an external field. In other words, substituting $\hext=0$ in Theorem \ref{rescaled.energy.compact} gives Theorem 9, so there is no need to state Theorem 9 here. However, it will be useful to name this theorem so we can return to it easily.
    \begin{theorem}\label{theorem9ignatk}
    Substitute $\hext=0$ in Theorem \ref{rescaled.energy.compact} to get the statement of this theorem (or see Theorem 9 in \cite{ignatk}).
    \end{theorem}
    \item Theorem \ref{rescaled+reduced}, which connects the three-dimensional energy $E_\alpha$ with the two-dimensional energy $\eeeb$.
    \item Some compactness and Gamma-convergence results for the energies $\Ge$ and $\eee$ given by \eqref{Gepsi.definition} and \eqref{epsilon.eta.energy.definition}, respectively.  
\end{itemize}

Therefore, let us state some compactness and $\Gamma$-convergence results for the energy $\Ge$, which are similar to the results of Theorem 4.2 in \cite{ignatk2}, but in the presence of an external field in this case. 
\begin{theorem}\label{Gepsilon.theorem}
 Let $\Omega \subset \R^{2}$ be a bounded, simply connected and $C^{1,1}$  domain.
\begin{enumerate}
  \item $L^{p}(\partial \Omega)$ \textbf{compactness and first order lower bound}: 
  Let $\left(\phie\right)_{\varepsilon}$ be a sequence / family in $H^{1}(\Omega ; \R)$ such that
\begin{equation*}
    \limsup _{\varepsilon \rightarrow 0} \frac{1}{|\log \varepsilon|} \Ge\left(\phie;\hex\right)<\infty, 
    \end{equation*}
then there is a sequence/ family $\left(z_{\varepsilon}\right)_{\varepsilon}$ of integers such that $\left(\phie-\pi z_{\varepsilon}\right)_{\varepsilon}$ converges (up to a subsequence) strongly in $L^{p}(\partial \Omega)$ to a limit $\phi_0$ such that $\phi_0-g \in B V(\partial \Omega ; \pi \mathbb{Z})$ with $g$ given above in \eqref{liftingtangent} and
\begin{equation*}
    \partial_{\tau} \phi_0=\varkappa-\pi \sum_{j=1}^{N} d_{j} \delta_{a_{j}}, \quad a_{j} \in \partial \Omega \  \text { distinct points, } d_{j} \in \mathbb{Z} \backslash\{0\} \quad \text { with } \sum_{j=1}^{N} d_{j}=2 ,
\end{equation*}
and $\partial_{\tau} \phie \rightarrow \partial_{\tau} \phi_0$ in $W^{-1, p}(\partial \Omega)$ for every $1 \leq p<\infty$. Furthermore, we have the following first-order lower bound
\begin{equation}\label{lowerbound1}
    \liminf _{\varepsilon \rightarrow 0} \frac{1}{|\log \varepsilon|} \Ge\left(\phie;\hex\right) \geq\left|\partial_{\tau} \phi_0-\varkappa\right|(\partial \Omega)=\pi \sum_{j=1}^{N}\left|d_{j}\right| .
\end{equation}
  \item $W^{1, q}(\Omega)$ \textbf{weak compactness and second order lower bound}: 
  Let $\left(\phie\right)_{\varepsilon}$ be a sequence/ family in $H^{1}(\Omega ; \R)$ satisfying the convergence at the previous point with the limit $\phi_0$ on $\partial \Omega$ as $\varepsilon \rightarrow 0$. If additionally we assume that
\begin{equation*}
    \limsup _{\varepsilon \rightarrow 0}\left(\Ge\left(\phie;\hex\right)-\pi|\log \varepsilon| \sum_{j=1}^{N}\left|d_{j}\right|\right)<\infty , 
\end{equation*}
then $d_{j} \in\{ \pm 1\}$ for all $j=1, \ldots, N, \left(\nabla \phie\right)_{\varepsilon}$ converges weakly (for a subsequence) in $L^{q}\left(\Omega ; \R^{2}\right)$ for any $q \in[1,2)$ to $\nabla \phihz$, where $\phihz \in W^{1, q}(\Omega)$ is an extension of $\phi_0$ to $\Omega$
. The following second-order lower bound holds for the sequence/ family $\varepsilon \rightarrow 0$ :
\begin{equation}\label{secondlowerbound}
    \liminf _{\varepsilon \rightarrow 0}\left(\Ge\left(\phie; \hex\right)-\pi N|\log \varepsilon|\right) \geq \Wad+N \gamma_{0},
\end{equation}
where $\W$ is the renormalized energy and $\gamma_{0}=\pi \log \frac{e}{4 \pi}$.
  \item \textbf{Upper bound construction}: Let $\phi_0: \partial \Omega \rightarrow \R$ be such that $\partial_{\tau} \phi_0=\varkappa-\pi \sum_{j=1}^{N} d_{j} \delta_{a_{j}}$, $d_{j} \in \mathbb{Z} \backslash\{0\}$ with $\sum_{j=1}^{N} d_{j}=2$, $e^{i \phi_0} \cdot \nu=0$ in $\partial \Omega \backslash\left\{a_{1}, \ldots, a_{N}\right\}$. Then for every $\varepsilon>0$ small, there exists $\phihe \in H^{1}(\Omega; \R)$ such that $\phihe \rightarrow \phi_0$ in $L^{p}(\partial \Omega)$ and $\phihe \rightarrow \sigmas + \phi_*$ in $L^{p}(\Omega)$ for every $p \in[1, \infty)$ where $\sigmas$ is a minimizer of 
  \begin{equation*}
    F(\sigma) = \int_\Omega \left(|\nabla \sigma|^2   - h_{ext}\cdot e^{i(\sigma+\phi_*)}\right) \, dx,
\end{equation*}
and $\phi_*$ is the function mentioned above in the definition of the unperturbed renromalized energy (Definition \ref{unperturbed.renormalized.def}).
   Moreover,
\begin{equation*}
    \limsup _{\varepsilon \rightarrow 0} \frac{1}{|\log \varepsilon|} \Ge\left(\phihe; \hex\right)=\pi \sum_{j=1}^{N}\left|d_{j}\right| . 
    \end{equation*}
If in addition $d_{j}= \pm 1$ for all $j$, then we have additionally
\begin{equation*}
\limsup _{\varepsilon \rightarrow 0}\left(\Ge\left(\phihe ; \hex\right)-N \pi \log \frac{1}{\varepsilon}\right)=\Wad+N \gamma_{0}.
\end{equation*}
\end{enumerate}
\end{theorem}

The proof of this theorem appears in section 3. Now, we state some compactness and Gamma-convergence results for the energy $\eee$. This theorem is similar to Theorem 18 in \cite{ignatk}, but again it has the effect of adding the external field. 
\begin{theorem}\label{epsilon.eta.energy}
Let $\Omega \subset \R^{2}$ be a bounded, simply connected $C^{1,1}$ domain, $\varepsilon \rightarrow 0$ and $\eta \rightarrow 0$ be sequences / families satisfying $|\log \varepsilon| \ll|\log \eta|$. Assume $u_{\varepsilon} \in H^{1}\left(\Omega ;\R^{2}\right)$ satisfy
\begin{equation}
\limsup _{\varepsilon \rightarrow 0} \frac{1}{|\log \varepsilon|} \eee\left(u_{\varepsilon}; \hex\right)<\infty .
\end{equation}
\begin{enumerate}
    \item \textbf{Compactness of global Jacobians and $L^{p}(\partial \Omega)$-compactness of $\left.u_{\varepsilon}\right|_{\partial \Omega}$}: For a subsequence, the global Jacobians $\mathcal{J}\left(u_{\varepsilon}\right)$ converge to a measure $J$ on the closure $\bar{\Omega}$ in the following sense 
    \begin{equation*}
       \sup_{|\nabla\xi|\leq 1  \ \text{in} \ \Omega} \left|\left\langle\mathcal{J}(u_\varepsilon) - J, \xi \right\rangle \right| \rightarrow 0 \ \ \text{as} \ \ \alpha\rightarrow 0 , \ \ \forall \xi \in W^{1,\infty}(\Omega), 
    \end{equation*}
   Moreover, $J$ is supported on $\partial \Omega$ and has the form 
   \begin{equation*}
       J = -\varkappa \mathcal{H}^1 \llcorner \partial\Omega + \pi \sum_{j=1}^{N}d_j \delta_{a_j}, \ \ \text{where} \ \ \sum_{j=1}^{N}d_j = 2 , 
   \end{equation*}
   for $N$ distinct boundary vortices $a_{j} \in \partial \Omega$ carrying the non-zero multiplicities $d_{j} \in \mathbb{Z} \backslash\{0\}$. Moreover, for a subsequence, the trace $\left.u_{\varepsilon}\right|_{\partial \Omega}$ converges as $\varepsilon \rightarrow 0$ in $L^{p}(\partial \Omega)$ for every $p \in[1, \infty)$ to $e^{i \phi} \in B V\left(\partial \Omega ; \mathbb{S}^{1}\right)$ for a lifting $\phi$ of the tangent field $\pm \tau$ on $\partial \Omega$ determined (up to a constant in $\pi \mathbb{Z}$ ) by $\partial_{\tau} \phi=-J$ on $\partial \Omega$.
\item \textbf{Energy lower bound at the first order}: If $\left(u_{\varepsilon}\right)$ satisfies the convergence assumption on the Jacobians as the sequence/family $\varepsilon \rightarrow 0$ as in the first point, then
\begin{equation*}
\liminf _{\varepsilon \rightarrow 0} \frac{1}{|\log \varepsilon|} \eee\left(u_{\varepsilon};\hex\right) \geq \pi \sum_{j=1}^{N}\left|d_{j}\right| =\mid J+\varkappa \mathcal{H}^{1}\llcorner\partial \Omega \mid(\partial \Omega) .
\end{equation*}
If, in addition, we assume the following sharper bound:
\begin{equation}\label{sharper.bound}
\limsup _{\varepsilon \rightarrow 0}\left(\eee\left(u_{\varepsilon};\hex\right)-\pi \sum_{j=1}^{N}\left|d_{j}\right||\log \varepsilon|\right)<\infty,
\end{equation}
then we get:
\item \textbf{Single multiplicity and second order lower bound}. The multiplicities satisfy $d_{j}= \pm 1$ for $1 \leq j \leq N$, so $\sum_{j=1}^{N}\left|d_{j}\right|=N$ and there holds the finer energy bound
\begin{equation*}
\liminf _{\varepsilon \rightarrow 0}\left(\eee\left(u_{\varepsilon}; \hex\right)-\pi N|\log \varepsilon|\right) \geq\Wad+\gamma_{0} N
\end{equation*}
where $\gamma_{0}=\pi \log \frac{e}{4 \pi}$ and $\W$ is the renormalised energy .
\item \textbf{Penalty bound}. The penalty terms are of order $O(1)$, i.e.
\begin{equation*}
\limsup _{\varepsilon \rightarrow 0}\left(\frac{1}{\eta^{2}} \int_{\Omega}\left(1-\left|u_{\varepsilon}\right|^{2}\right)^{2} d x+\frac{1}{2 \pi \varepsilon} \int_{\partial \Omega}\left(u_{\varepsilon} \cdot \nu\right)^{2} d \mathcal{H}^{1}-\int_{\Omega} \hex \cdot u_\varepsilon dx\right)<\infty
\end{equation*}
\item \textbf{Local energy lower bound}: There are $\rho_{0}>0, \varepsilon_{0}>0$ and $C>0$ such that the energy of $u_{\varepsilon}$ near the singularities satisfies for all the $\varepsilon<\varepsilon_{0}$ in the sequence / family and $\rho<\rho_{0}$ :
\begin{equation*}
\left(\int_{\Omega \cap \bigcup_{j=1}^{N} B_{\rho}\left(a_{j}\right)}\left|\nabla u_{\varepsilon}\right|^{2} d x-\pi N \log \frac{\rho}{\varepsilon}\right)>-C .
\end{equation*}
\item \textbf{$L^{p}(\Omega)$-compactness of maps $u_{\varepsilon}$}: For any $q \in[1,2)$, the sequence/family $\left(u_{\varepsilon}\right)_{\varepsilon}$ is uniformly bounded in $W^{1, q}\left(\Omega ; \R^{2}\right)$. Moreover, for a subsequence, $u_{\varepsilon}$ converges as $\varepsilon \rightarrow 0$ strongly in $L^{p}\left(\Omega ; \R^{2}\right)$ for any $p \in[1, \infty)$ to $e^{i \hat{\phi}}$, where $\hat{\phi} \in W^{1, q}(\Omega)$ is an extension to $\Omega$ of the lifting $\phi \in B V(\partial \Omega)$ determined in the first point.

\item \textbf{Matching upper-bound construction:} Given any $N$ distinct points $a_{j} \in \partial \Omega$ with their multiplicity $d_{j} \in \mathbb{Z} \backslash\{0\}$ satisfying the constraint $\sum_{j=1}^{N} d_{j}=2$, we can construct for every $\varepsilon \in\left(0, \frac{1}{2}\right), \
 u_{\varepsilon} \in H^{1}\left(\Omega ; \mathbb{S}^{1}\right)$ such that the global Jacobians $\mathcal{J}\left(u_{\varepsilon}\right)$ converge to \begin{equation*}
     J=-\varkappa \mathcal{H}^{1}\left\llcorner\partial \Omega+\pi \sum_{j=1}^{N} d_{j} \delta_{a_{j}}\right.
 \end{equation*} as in part 1 of this theorem. Furthermore, $u_{\varepsilon}$ converge strongly to a canonical harmonic map $m_{*}$ associated to $\left\{\left(a_{j}, d_{j}\right)\right\}$ in $L^{p}(\Omega)$ and $L^{p}(\partial \Omega)$ for all $p \in[1, \infty)$, and the energies satisfy
\begin{equation}
\lim _{\varepsilon \rightarrow 0} \frac{1}{|\log \varepsilon|} \eee\left(u_{\varepsilon};\hex\right)=\pi \sum_{j=1}^{N}\left|d_{j}\right|.
\end{equation}
If furthermore $\left|d_{j}\right|=1$ for all $j=1, \ldots, N$, then $u_{\varepsilon}$ can be chosen such that
\begin{equation}
\lim _{\varepsilon \rightarrow 0}\left(\eee\left(u_{\varepsilon};\hex\right)-\pi N|\log \varepsilon|\right)=\Wad+N \gamma_{0} .
\end{equation}
\end{enumerate}
\end{theorem} 
Note: As defined in \cite{ignatk}, a \textit{canonical harmonic map} here means an $\Ss^1$ valued smooth harmonic map $m_*=e^{i\phi_*}$ (where $\phi_*$ is harmonic in $\Omega$) that is tangent on the boundary of the sample $\partial\Omega$ except at $N$ boundary vortices $a_j\in\partial\Omega$, where $m_*$ winds at each vortex $a_j$ according to the multiplicity $d_j$ for $j=1,2,...,N$. 
\subsection{Other thin-film regimes}
The variational problem we study here depends on many parameters like $\alpha$, $\eta$, and $\varepsilon$, which makes the problem rich as it can be studied in different regimes and we always assume $\alpha\ll 1$ as we only consider thin-film regimes. We already mentioned the regime we are working on \eqref{regime}, which was studied by Ignat and Kurzke \cite{ignatk} in the absence of an external field. They proved the concentration of the energy around boundary vortices by a $\Gamma$-convergence expansion of the three dimensional energy at the second order. Baffetti \cite{marco} also worked in the same regime and studied the micromagnetic properties of rectangular thin-films and showed that the S-state magnetizations have minimal energy.

The first rigorous analysis of a micromagnetic thin-film limit was done by Gioia and James \cite{gioia} as they studied this problem when $\eta$ is fixed (see also \cite{Kreisbeck}). They showed that the $\Gamma$-limit is minimized by all constant in-plane magnetizations and that these magnetizations do not depend on the shape of the film. Kohn and Slastikov \cite{kohn} studied the problem for relatively small soft thin-films when $\eta^2 \gg \alpha|\log\alpha|$. In this regime, the exchange energy becomes the dominant term in the micromagnetic energy and the non-local magnetostatic energy term reduces to the local term $\int_{\partial\Omega}(m\cdot\nu)^2 d\mathcal{H}^1$. Moreover, the minimizing magnetization $m$ is a constant unit-length vector, independent of the thickness of the film, and has no out-of-plane component. They also studied slightly larger films where $\eta^2 \sim  \alpha|\log\alpha|$ and showed that the rescaled micromagnetic energy $\Gamma$-converges to the energy
\begin{equation*}
    E_{KS}^\beta := \beta \int_{\Omega} |\nabla m|^2 \,dx + \frac{1}{2\pi} \int_{\partial\Omega} (m\cdot\nu)^2 \, d\mathcal{H}^1 \ , \ \ m\in H^1(\Omega;\mathbb{S}^1) ,
\end{equation*}
where $\beta= \lim_{\alpha\rightarrow 0} \frac{\eta^2}{\alpha|\log\alpha|}$. In this case, the limiting magnetization is a unit in-plane vector as before, but not necessarily a constant and the non-local term still reduces to the local term $\int_{\partial\Omega}(m\cdot\nu)^2 d\mathcal{H}^1$ times a constant. Kurzke \cite{kurzke} studied the behaviour of $\frac{1}{\beta}E_{KS}^\beta$ when $\beta \rightarrow 0$ and showed that there is no magnetization $m$ in $H^1(\Omega;\mathbb{S}^1)$ which satisfies $m\cdot\nu = 0$ on the boundary of the sample $\partial\Omega$ when $\Omega$ is simply connected. Therefore, the boundary term of $\frac{1}{\beta}E_{KS}^\beta$ cannot be made zero, so we obtain the formation of boundary vortices, where the magnetization quickly changes from $m=\tau$ to $m=-\tau$ over a segment of $\partial\Omega$ of order $O(\beta)$ (see also \cite{kurzke1, gradientflow}).

Moser \cite{moser} studied the regime $\eta^2 = O(\alpha)$ and showed that both the exchange and the magnetostatic terms are present in the energy limit in this regime, and boundary vortices will appear (see also \cite{moser2, moser3}). Ignat and Otto \cite{otto} studied the regime $\eta^2 \ll \alpha$, which corresponds to large thin-films and showed that the stable state magnetizations in this regime consist of magnetic domains separated by Néel walls, and an interior vortex or two boundary vortices appear in this regime.  Ignat and Kn{\"u}pfer \cite{knupfer} studied the magnetization pattern for thin ferromagnetic films with circular cross-section, i.e. $\Omega=B_\ell \times (0,t)$ where $B_\ell$ is the disk in $\R^2$ with radius $\ell$ and $t$ is the thickness of the film. They studied the problem in the regime
\begin{equation*}
 \frac{\alpha}{|\log\alpha|} \ll \eta^2 \left|\log\frac{\eta^2}{\alpha}\right|  \ll \frac{\alpha}{\log|\log\alpha|}
\end{equation*}
and proved that the vortex structure in this regime is driven by a $360^{\circ}$–N\'{e}el wall with two boundary vortices.

DeSimone et al. \cite{simone, recent} studied the regime $\eta^2 \ll \frac{\alpha}{|\log\alpha|}$ which corresponds to very large thin-films. In this case, the magnetostatic energy will be the dominant term in the micromagnetic energy and the contribution of the exchange term disappears. L'Official \cite{francois} studied boundary vortices in a thin-film regime where the micromagnetic energy has an additional term called the Dzyaloshinskii-Moriya interaction term (see also \cite{ignat+Lofficial}). Di Fratta, Muratov and Slastikov \cite{fratta} studied the problem when the energy has a perpendicular anisotropy term and a Dzyaloshinskii-Moriya interaction term and they worked in four different asymptotic regimes.

\subsection*{Outline of the paper}  In the next section, we first prove Lemma \ref{renormalized_lemma} that relates the new definition of the renormalized energy in the presence of an external field with the energies of the form of $G$ given by \eqref{Genergy}. Then we prove the compactness and $\Gamma$-convergence results for the energy $\Ge$ (Theorem \ref{Gepsilon.theorem}). In section 4, we prove Theorem \ref{epsilon.eta.energy} that has compactness and $\Gamma$-convergence results for the energy $\eee$. Then, we prove Theorem \ref{rescaled+reduced} that connects the 3D energy with the 2D one. Finally, we prove the compactness and $\Gamma$-convergence results for the energy $E_\alpha$ (Theorem \ref{rescaled.energy.compact}). In the last section, we illustrate the magnetization vector field behavior in the unit disk and an oval-shaped domain and show the effect of adding different external fields on the magnetization vector field and the boundary vortices' locations in these domains. 

\subsection*{Acknowledgment} 
This paper presents work from the first author's PhD thesis, and she would like to thank the School of Mathematical Sciences at the University of Nottingham for supporting her studies. 

\section{Gamma-convergence for \texorpdfstring{$\Ss^1$-}{S1-}valued magnetizations (Theorem \ref{Gepsilon.theorem})}
In this section, we mainly prove Theorem \ref{Gepsilon.theorem}, but before that, let us prove Lemma \ref{renormalized_lemma}. 
\begin{proof} (Of Lemma \ref{renormalized_lemma})
We defined the renormalized energy as  
\begin{equation*}
    \Wad = W_{\Omega;0}(\{ (a_j,d_j)\}) + F(\sigmas),
\end{equation*}
where $ W_{\Omega;0}(\{ (a_j,d_j)\})$ is the unperturbed renormalized energy and $\sigmas$ is any minimizer of  
\begin{equation*}
    F(\xi) = \int_\Omega \left(|\nabla \xi|^2   - h_{ext}\cdot e^{i(\xi+\phi_*)}\right) \, dx,
\end{equation*}
where $\phi_*$ is the solution of 
\begin{equation}\label{kappas_equation}
  \begin{cases}
  \Delta\phi_*  = 0 \ &\text{in} \ \Omega,\\
    \phi_*  = \phib \  & \text{on} \ \partial\Omega.  
  \end{cases} 
\end{equation}
In this lemma, we want to prove that for $\psi\in H^{1}_{loc}(\bar{\Omega}\setminus \{a_1,...,a_N\})$ where $\psi=\phib$ on $\partial\Omega$ (where $\phib$ is the one mentioned in Definition \ref{unperturbed.renormalized.def} ), the following inequality holds:
\begin{equation*}
    G(\psi) = \liminf_{\rho \to 0}\left(  \int_{\Omega\backslash \cup_{j=1}^{N} B_\rho(a_j)}|\nabla \psi|^2 \,dx   - \pi N \log \frac{1}{\rho}\right) - \int_{\Omega} h_{ext}\cdot e^{i\psi} \, dx \geq \Wad.
\end{equation*}

To do that, we write $\psi $ as the sum $\psi=\phi_*+\sigma$, where $\phi_*$ is defined above.
Then, the first term of $G(\psi)$ is
\begin{equation*}
    \int_{\Omega_\rho}|\nabla \psi|^2 \,dx =    \int_{\Omega_\rho}|\nabla \phi_*|^2 \,dx +   \int_{\Omega_\rho}|\nabla \sigma|^2 \,dx +2   \int_{\Omega_\rho}\nabla \phi_* \cdot \nabla\sigma \,dx, 
\end{equation*}
where $\Omega_\rho = \Omega\backslash \cup_{j=1}^{N} B_\rho(a_j)$. By Green's formula, the last term above is 
\begin{align*}
    2   \int_{\Omega_\rho}\nabla \phi_* \cdot \nabla\sigma \,dx &= -2  \int_{\Omega_\rho} (\Delta\phi_*) \sigma \, dx + 2  \int_{\partial\Omega_\rho} \frac{\partial\phi_*}{\partial\nu} \sigma \, dS, \\
    &= 2  \int_{\partial\Omega_\rho} \frac{\partial\phi_*}{\partial\nu} \sigma \, dS , \ \ \text{(since $\Delta\phi_*  = 0$ in $\Omega$)}.
\end{align*}
But $\sigma=0$ on $\partial\Omega$ (since $\psi=\phi_*=\phib$ on $\partial\Omega$), so
\begin{equation*}
  2  \int_{\partial\Omega_\rho} \frac{\partial\phi_*}{\partial\nu} \sigma \, dS = 2  \int_{\Omega\cap\cup_{j=1}^{N}\partial  B_\rho(a_j)} \frac{\partial\phi_*}{\partial\nu} \sigma  \, dS
\end{equation*}
Now, if we show that:
\begin{equation*}
    \int_{\Omega\cap\cup_{j=1}^{N}\partial  B_\rho(a_j)} \left|\sigma\right|^2  \, dS\leq C_1 \rho   \int_{\Omega_\rho} |\nabla\sigma|^2 \, dx,
\end{equation*}
and if we show that  $\int_{\Omega\cap\cup_{j=1}^{N}\partial  B_\rho(a_j)} \left|\frac{\partial\phi_*}{\partial\nu}\right|^2 \, dS=O(\rho (\log \rho)^2)$, then using these results and H\"{o}lder's inequality, we get 
\begin{align*}
   \left|   2  \int_{\Omega\cap\cup_{j=1}^{N}\partial  B_\rho(a_j)} \frac{\partial\phi_*}{\partial\nu} \sigma \, dS  \right|& \leq   2\left(  \int_{\Omega\cap\cup_{j=1}^{N}\partial  B_\rho(a_j)} \left|\frac{\partial\phi_*}{\partial\nu}\right|^2 \, dS\right) ^{1/2} \left(  \int_{\Omega\cap\cup_{j=1}^{N}\partial  B_\rho(a_j)} \left|\sigma\right|^2 \, dS\right) ^{1/2}, \\
      & \leq 2 \left(C\rho^{1/2} \log \rho\right) \left(C_1 \rho \int_{\Omega_\rho} |\nabla\sigma|^2 \, dx\right)^{1/2} .
\end{align*}
Therefore, as $\rho$ goes to zero, we get 
\begin{equation*}
    \lim_{\rho \to 0}{\left|2  \int_{\Omega\cap\cup_{j=1}^{N}\partial  B_\rho(a_j)} \frac{\partial\phi_*}{\partial\nu} \sigma \, dS\right|} =0.
\end{equation*}
This implies that the limit of the first term of $G(\psi)$ is
\begin{equation*}
   \lim_{\rho\to 0} \int_{\Omega_\rho}|\nabla \psi|^2 \,dx =     \lim_{\rho\to 0} \int_{\Omega_\rho}|\nabla \phi_*|^2 \,dx +    \lim_{\rho\to 0} \int_{\Omega_\rho}|\nabla \sigma|^2 \,dx. 
\end{equation*}
Hence, 
\begin{align*}
    G(\psi)=& \liminf_{\rho \to 0}\left(\int_{\Omega_\rho}|\nabla \psi|^2 \,dx   - \pi N \log \frac{1}{\rho}\right) - \int_{\Omega} h_{ext}\cdot e^{i\psi} \, dx \\
     = &  \liminf_{\rho \to 0} \left( \int_{\Omega_\rho}|\nabla \phi_*|^2 \,dx +   \int_{\Omega_\rho}|\nabla \sigma|^2 \,dx - \pi N \log \frac{1}{\rho}\right) - \int_{\Omega} h_{ext}\cdot e^{i(\phi_*+\sigma)} \, dx,\\
     =&  \liminf_{\rho \to 0} \left( \int_{\Omega_\rho}|\nabla \phi_*|^2 \,dx - \pi N \log \frac{1}{\rho}\right) +  \liminf_{\rho \to 0}\int_{\Omega_\rho}|\nabla \sigma|^2 \,dx - \int_{\Omega} h_{ext}\cdot e^{i(\phi_*+\sigma)} \, dx,\\
     = &  W_{\Omega;0}(\{ (a_j,d_j)\})+\liminf_{\rho \to 0}\int_{\Omega_\rho}|\nabla \sigma|^2 \,dx - \int_{\Omega} h_{ext}\cdot e^{i(\phi_*+\sigma)} \, dx,\\
     \geq &  W_{\Omega;0}(\{ (a_j,d_j)\}) + F(\sigmas) = \Wad
\end{align*}
Note that in the last two steps above, we used that $\liminf_{\rho \to 0}\int_{\Omega_\rho}|\nabla \sigma|^2 \, dx = \int_{\Omega}|\nabla \sigma|^2 \,dx$. This is true here because:  $\int_{\Omega_\rho}|\nabla\sigma|^2\leq G(\psi)+C$ for some constant $C$ (note that $G(\psi)$ is bounded, since otherwise there is nothing to prove in this lemma).
Hence, $\int_{\Omega_\rho}|\nabla\sigma|^2\leq C'$ for some $C'\in \R$, and by the monotone convergence theorem, we get $\lim_{\rho \to 0}\int_{\Omega_\rho}|\nabla \sigma|^2 \, dx = \int_{\Omega}|\nabla \sigma|^2 \,dx$.


To finish the proof, we need to prove the following two claims.

Claim 1:
Let $\phi_*$ be the solution of 
\begin{equation}\label{claimpde}
  \begin{cases}
  \Delta\phi_*  = 0 \ &\text{in} \ \Omega,\\
    \phi_*  = \phib \  & \text{on} \ \partial\Omega.  
  \end{cases} 
\end{equation}
Then $\int_{\Omega\cap\cup_{j=1}^{N}\partial  B_\rho(a_j)} \left|\frac{\partial\phi_*}{\partial\nu}\right|^2 \, dS=O(\rho \left(\log \rho\right)^2)$. 
\begin{proof}

  Let $\Phi:\bar{\omega}\to \bar{\Omega}$ be a $C^1$ conformal diffeomorphism between two simply connected domains $\omega , \Omega\subset\R^2$, and let $\Psi$ be its inverse. Let us denote the solution of equation \eqref{claimpde} in $\R_+^2$ by $\Tilde{\phi}_{*}$. Then $\Tilde{\phi}_*(z)=\sum_{j=1}^{N} d_j \text{Arg}(z-a_j)$ for $z\in \R_+^2$ is a solution of \eqref{claimpde} in $\R_+^2$ up to an additive constant in $\pi\mathbb{Z}$.
Hence, $\phi_*$ in $\Omega$ can be written as 
    \begin{equation*}
        \phi_*(w)=\Tilde{\phi}_*(\Psi(w)) +\Theta(\Psi(w)), \ \ \forall w \in \Omega,
    \end{equation*}
where $\Theta$ is a smooth harmonic function from $\omega$ to $\R$ such that $e^{i\Theta(z)}= \Phi'(z)/|\Phi'(z)|$ for all $z\in \omega$.

We want to prove the claim in a bounded simply connected $C^{1,1}$ domain $\Omega$. To find an upper bound for the gradient of $u(z)=\phi_*(z)- \sum_{j=1}^{N} d_j \text{Arg}(z-a_j)$ for all $z$ in $\Omega$, we use Theorem 3 in \cite{Hile}. 
\begin{theorem}\label{Theorem3Hile} (Theorem 3 in \cite{Hile} ) Let $\Omega$ be a bounded Dini-smooth Jordan domain in the plane, 
let $\phi:\partial\Omega\to \R$ be Lipschitz continuous with Lipschitz constant $M$, and let 
$u$ solve the Dirichlet problem $u \in C(\bar{\Omega})$, $\Delta u=0$ in $\Omega$, $u = \phi$ on $\partial\Omega$. Then there exists a positive constant $C(\Omega)$ such that, for all $z$ in $\Omega$,
\begin{equation}
    |\nabla u(z)| \leq C(\Omega) M \log \frac{d(\Omega)}{d(z,\partial\Omega)},
\end{equation}
    where $d(\Omega)$ is the diameter of $\Omega$ and $d(z,\partial\Omega)$ is the distance from $z$ to $\partial\Omega$.
\end{theorem}

Now, since $\Omega$ is a bounded, simply connected,  $C^{1,1}$ domain, it is a bounded Dini-smooth Jordan domain. Moreover, the function $u(z)=\phi_*(z)- \sum_{j=1}^{N} d_j \text{Arg}(z-a_j)$ satisfies the assumptions of Theorem \eqref{Theorem3Hile} , i.e. $\Delta u =0$ in $\Omega$ 
and $u$ is Lipschitz continuous on the boundary.
Therefore, using Theorem \eqref{Theorem3Hile} we get the following bound:
\begin{equation}\label{gradientbound}
    |\nabla u(z)| \leq C(\Omega) M \log \frac{d(\Omega)}{d(z,\partial\Omega)}, \ \ \ \forall z\in \Omega.
\end{equation}
\begin{align}
 \int_{\Omega\cap\cup_{j=1}^{N}\partial  B_\rho(a_j)} \left|\frac{\partial \phi_*}{\partial\nu}\right|^2 \, dS   &= \sum_{j=1}^N\int_{\Omega\cap\partial  B_\rho(a_j)} \left|\frac{\partial \phi_*}{\partial\nu}\right|^2 \, dS , \nonumber\\
 &=\sum_{j=1}^N\int_{\Omega\cap\partial  B_\rho(a_j)} \left|\frac{\partial u(z)}{\partial\nu} + d_j \frac{\partial }{\partial\nu}\text{Arg}(z-a_j) + \sum_{k\neq j}
 d_k  \frac{\partial }{\partial\nu} \text{Arg}(z-a_k)\right|^2 \, dS , \nonumber\\
  &=\sum_{j=1}^N\int_{\Omega\cap\partial  B_\rho(a_j)} \left|\frac{\partial u(z)}{\partial\nu} + \sum_{k\neq j}
 d_k  \frac{\partial }{\partial\nu} \text{Arg}(z-a_k)\right|^2 \, dS , \nonumber\\
 &\leq \sum_{j=1}^N\int_{\Omega\cap\partial  B_\rho(a_j)} \left( 2 \left|\frac{\partial u(z)}{\partial\nu}\right|^2 + 2\left|\sum_{k\neq j}
 d_k  \frac{\partial }{\partial\nu} \text{Arg}(z-a_k)\right|^2  \right), \label{normal.der}
 \end{align}
where in the third line above we used that for $z\in \Omega\cap\partial  B_\rho(a_j)$, the normal derivative of the Argument function is zero, i.e. $\nu.\nabla \text{Arg}(z-a_j)=0$ on $\partial  B_\rho(a_j)$. In addition, in the last line above, we used the following inequality: $|a+b|^2\leq 2|a|^2 + 2|b|^2$.
 
Let us find a bound for each term in \eqref{normal.der}. For $z \in \Omega\cap\partial  B_\rho(a_j)$, the second term above is bounded since 
 \begin{align}
     \left| \frac{\partial }{\partial\nu} \text{Arg}(z-a_k)\right| &= \left| \nu \cdot\nabla \text{Arg}(z-a_k)\right|,\nonumber\\
     &\leq  \left| \nabla \text{Arg}(z-a_k)\right|, \ \ \  \text{(since $\nu$ is a unit vector)}, \nonumber\\
     &\leq \frac{1}{|z-a_k|}, 
 \end{align}
 but $|z-a_k|=|z-a_j+a_j-a_k| \geq |a_j-a_k|-\rho$ by the triangle inequality
 . Moreover, there is a constant, say $\delta_*>0$ such that $|a_j-a_k|-\rho\geq \delta_*$ 
 . Hence, we get that $\left| \frac{\partial }{\partial\nu} \text{Arg}(z-a_k)\right|$ is bounded. The boundedness of this directional derivative, together with Cauchy-Schwarz inequality, implies that the second term in \eqref{normal.der} is bounded.

Now, the first term in \eqref{normal.der} is bounded since
\begin{align}
\int_{\Omega\cap\partial  B_\rho(a_j)}  \left|\frac{\partial u(z)}{\partial\nu}\right|^2  &= \int_{\Omega\cap\partial  B_\rho(a_j)} \left|\nu \cdot \nabla u\right|^2 \, dS\nonumber\\
    &\leq\int_{\Omega\cap\partial  B_\rho(a_j)} \left|\nabla u\right|^2 \, dS,  \ \ \ \text{(since $\nu$ is a unit vector),} \nonumber\\
     &\leq C(\Omega) M^2 \int_{\Omega\cap\partial  B_\rho(a_j)} \left|\log \frac{d(\Omega)}{d(z,\partial\Omega)}\right|^2\, dS,  \ \ \ \text{(using inequality \eqref{gradientbound}).}\label{distance.integ}
     \end{align}
Note that $d(\Omega)$ is bounded since $\Omega$ is a bounded domain. Now, we want to find a lower bound for $d(z,\partial\Omega)$  
for $z\in \Omega\cap\partial  B_r(a_j)$ for some $j$. To do that, we write the point $z$ as $z=re^{i\theta}$ for some $\theta\in(\theta_{j1}, \theta_{j2})$, where $\theta_{j1}, \theta_{j2} \in[0,2\pi)$. The two points $z_1=re^{i\theta_{j1}}$ and $z_2=re^{i\theta_{j2}}$ are the intersection points between $\partial\Omega$ and $\partial  B_r(a_j)$. Let $\beta_1$, $\beta_2$, and $\beta_3$ be the three angles in the triangle with vertices $z=re^{i\theta}$, $z_1=re^{i\theta_{j1}}$
 and $p$ which is the projection of $z$ onto $\partial\Omega$. Then,
 \begin{equation*}
     \frac{\sin\beta_1}{d(re^{i\theta}, re^{i\theta_{j1}})}=\frac{\sin\beta_2}{d(re^{i\theta}, p)}=\frac{\sin\beta_3}{d(re^{i\theta_{j1}}, p)} .
 \end{equation*}
For sufficiently small $r<r_0$, where $0<r_0\in \R$, there exists constants $c_0,c_1\in \R$ such that $\frac{\pi}{2}-c_0 \leq \beta_1 \leq \frac{\pi}{2}+c_0 $ and $\frac{\pi}{4}-c_1 \leq\beta_2$ (if $\beta_2< \frac{\pi}{4}+c_0 $ then consider the triangle with vertices $re^{i\theta}$, $re^{i\theta_{j2}}$, and $p$). This implies that there is a constant $c_2\in \R$ such that
\begin{equation}\label{distance.ineq}
    d(z,\partial\Omega)= d(re^{i\theta},p)= \frac{\sin\beta_2}{\sin\beta_1}d(re^{i\theta}, re^{i\theta_{j1}})\geq c_2 d(re^{i\theta}, re^{i\theta_{j1}})=2c_2 r\left|\sin\frac{\theta-\theta_{j1}}{2}\right|.
\end{equation}
From equations \eqref{distance.integ} and \eqref{distance.ineq}, we get
\begin{align*}
 \int_{\Omega\cap\partial  B_\rho(a_j)} \left|\frac{\partial u}{\partial\nu}\right|^2 \, dS   &\leq C(\Omega)M^2\int_{\Omega\cap\partial  B_\rho(a_j)}  \left|\log \frac{d(\Omega)}{d(z,\partial\Omega)}\right|^2\, dS,\\
 & \leq C_1 \int_{\theta_{j1}}^{\pi/2}  \left|\log \frac{d(\Omega)}{2c_2 r\sin\frac{\theta-\theta_{j1}}{2}}\right|^2\,  r \, d\theta \\
 &+  C_2 \int_{\pi/2}^{\theta_{j2}}  \left|\log \frac{d(\Omega)}{2c_2 r\sin\frac{\theta_{j2}-\theta}{2}}\right|^2\,  r \, d\theta  .
\end{align*}
Each of the above integrals is similar to the following integral
\begin{equation*}
    I=\int_0^{\pi/2} \left(\log\frac{1}{r \sin\theta}\right)^2 \, r\, d\theta,
\end{equation*}
Therefore, it is enough to show that $I$ is bounded to obtain the result.
\begin{align}
   I&=   \int_0^{\pi/2}  \left(\log\frac{1}{r \sin\theta}\right)^2 \, r\, d\theta,  \nonumber\\
   & = \int_0^{\pi/2}  \left(\log r + \log\sin\theta\right)^2 \, r\, d\theta, \nonumber\\
&=  \frac{\pi}{2} r  \left(\log r \right)^2  + 2 r \log r \int_0^{\pi/2}\log \sin \theta \, d\theta +  r \int_0^{\pi/2} \left(\log \sin \theta\right)^2 \, d\theta\nonumber.  \label{polar.R2}
\end{align}
But $\int_0^{\pi/2}\log \sin \theta \, d\theta=-\frac{\pi}{2} \log 2$, and $\int_0^{\pi/2} \left(\log \sin \theta\right)^2 \, d\theta=\frac{\pi}{2}[(\log2)^2+ \frac{\pi^2}{12}]$ (these integrals can be found, for example, in Gradshteyn and Ryzhik \cite{Ryzhik} ). Hence, $I=O(r(\log r)^2)$, which ends the proof of the claim.


\end{proof}
Now, let us prove the second claim. 

Claim 2:  
\begin{equation*}
    \int_{\Omega\cap\cup_{j=1}^{N}\partial  B_\rho(a_j)} \left|\sigma\right|^2  \, dS\leq C_1 \rho   \int_{\Omega_\rho} |\nabla\sigma|^2 \, dx,
\end{equation*}
\begin{proof}
  First, we will rewrite and prove the claim for the half-space.

  Claim: let $\Omega =\R^2_+$, then there is a constant $C>0$ and $\rho_0>0$ such that for all $\rho<\rho_0$ and for all $u\in H^1(\R^2_+\cap(B_{2\rho}\setminus B_\rho))$, where $u=0$ on $\R \cap (B_{2\rho}\setminus B_\rho)$, the following inequality holds: 
 \begin{equation*}
    \int_{\R^2_+\cap \partial B_\rho} \left|u\right|^2  \, dS\leq C \rho   \int_{\R^2_+ \cap(B_{2\rho}\setminus B_\rho)} |\nabla u|^2 \, dx.
\end{equation*}
\begin{proof}
By rescaling, it is enough to prove it for $\rho=1$ as the proof is the same for any other $\rho$. Hence, let $\rho=1$ and by contradiction, assume there is no such $C\in \R$ where the claim holds. Then there is a sequence $u_n \in H^1(\R^2_+ \cap(B_{2}\setminus B_1))$ where $u_n=0$ on $\R\cap(B_{2}\setminus B_1)$ such that 
  \begin{equation*}
    \int_{\R^2_+\cap \partial B_1} \left|u_n\right|^2  \, dS > n   \int_{\R^2_+ \cap(B_{2}\setminus B_1)} |\nabla u_n|^2 \, dx.
\end{equation*}
Now, let $$\hat{u}_n= \frac{u_n}{\sqrt{  \int_{\R^2_+\cap \partial B_1} \left|u_n\right|^2  \, dS}},$$ then 
\begin{equation*}
    \int_{\R^2_+ \cap(B_{2}\setminus B_1)} |\nabla \hat{u}_n|^2 \, dx \leq \frac{1}{n},
\end{equation*}
which implies that $\nabla \hat{u}_n \rightarrow 0$ as $n \rightarrow\infty$ in $L^2(A)$, where $A=\R^2_+ \cap(B_{2}\setminus B_1)$. 

Now, set $a_n=\frac{1}{|A|}\int_A \hat{u}_n dx$. Then, by Poincaré inequality, we get $\|\hat{u}_n-a_n\|_{L^2(A)}\to0$, so $v_n:=\hat{u}_n-a_n \to 0$ in $H^1(A)$. 
By compact trace embedding, $v_n \to 0$ in $L^2(\partial A)$,  so $a_n\to 0$ (since $0=\int_{\partial A} |v_n|^2 = \int_{\partial A} |\hat{u}_n|^2+\int_{\partial A} |a_n|^2=0 +\int_{\partial A} |a_n|^2 $, so $\int_{\partial A} |a_n|^2=0$. But $a_n$ is a constant, so $a_n=0$).
 This implies that $\hat{u}_n \to 0$ in $H^1(A)$.


Then, 
\begin{equation*}
    1=\int_{\R^2_+\cap \partial B_1} \left|\hat{u}_n\right|^2 \, dS \to  0,
\end{equation*}
which is a contradiction. This ends the proof of the claim in $\R^2_+$.
\end{proof}

Next, we want to prove the previous claim when $\Omega$ is a $C^{1,1}$ domain. To do so, we need to prove the following.

Claim: Let $\Omega$ be a $C^{1,1}$ domain. Then there exist $C(\Omega)>0$ and  $\rho_0(\Omega)>0$  depending on the domain $\Omega$ such that for all $\rho<\rho_0(\Omega)$ there is a diffeomorphism 
\begin{equation*}
    \Psi_\rho : \R^2_+\cap(B_{2\rho}\setminus B_\rho) \to \Omega\cap(B_{2\rho}\setminus B_\rho), 
\end{equation*}
where the determinant of the Jacobian of $ \Psi_\rho$ and its inverse are bounded by $C(\Omega)$.
\begin{proof}
    Let us construct the diffeomorphism $\Psi_\rho$ that takes $(r,\theta) \in \R^2_+\cap(B_{2\rho}\setminus B_\rho)$ into $(s,t) \in \Omega\cap(B_{2\rho}\setminus B_\rho)$, where $s, r \in [r_0,2r_0]$ for some $r_0\in \R$, $\theta\in[0,\pi]$ and $t\in [t_1(r),t_2(r)]$, by mapping $(r,\theta)$ into $(s,t)= (r, t_1(r)+(t_2(r)-t_1(r))\frac{\theta}{\pi})$. 
    
    Calculating the Jacobian of this transformation, we get:
\begin{equation*}
    s_r=1, \ \ s_\theta=0, \ \ t_r= t_1'(r)+(t_2'(r)-t_1'(r))\frac{\theta}{\pi}, \ \ \text{and} \ \ t_\theta= \frac{t_2(r)-t_1(r)}{\pi}.
\end{equation*}
The determinant of the Jacobian matrix of $\Psi_\rho$ is $\frac{t_2(r)-t_1(r)}{\pi}$, which will be in the interval $[\frac{1}{2},\frac{3}{2}]$ 
if the balls $B_\rho$ are small enough. Hence, it is bounded by a constant, say $C(\Omega)>0$. Moreover, the determinat of the Jacobian inverse will be bounded since it is equal to $\frac{\pi}{t_2(r)-t_1(r)}\in[\frac{2}{3},2]$. 

\end{proof}
Now, let us prove the inequality when $\Omega$ is a $C^{1,1}$ domain. 

Claim: let $\Omega$ be a $C^{1,1}$ domain, then there is a constant $C>0$ and $\rho_0>0$ such that for all $\rho<\rho_0$ and for all $u\in H^1(\Omega\cap(B_{2\rho}\setminus B_\rho))$, where $u=0$ on $\partial\Omega \cap (B_{2\rho}\setminus B_\rho)$, the following inequality holds: 
 \begin{equation*}
    \int_{\Omega\cap \partial B_\rho} \left|u\right|^2  \, dS\leq C \rho   \int_{\Omega \cap(B_{2\rho}\setminus B_\rho)} |\nabla u|^2 \, dx.
\end{equation*}
\begin{proof}
    Using the diffeomorphism constructed above, let $v=u \circ \Psi_\rho$. Then, as $u=v \circ \Psi_\rho^{-1}$,  we can write the LHS of the above inequality as
     \begin{align*}
           \int_{\Omega\cap \partial B_\rho} \left|u\right|^2  \, dS &=  \int_{\Omega\cap \partial B_\rho} \left|v\circ \Psi_\rho^{-1}\right|^2 \, dS,\\
           & = \int_{\R^2_+\cap \partial B_\rho} \left|v\right|^2 \, dS,\ \ \ (\text{by a change of variables}) 
           \\
           &\leq  \, C \rho   \int_{\R_+^2 \cap(B_{2\rho}\setminus B_\rho)} |\nabla v|^2 \, dx,\ \ \ (\text{from the claim in the half space})  \\
           &=C \rho   \int_{\Omega \cap(B_{2\rho}\setminus B_\rho)} |\nabla v\circ \Psi_\rho^{-1}|^2 \, dx, \ \ \ (\text{by another change of variables})\\
           &=C \rho   \int_{\Omega \cap(B_{2\rho}\setminus B_\rho)} |D\Psi_\rho \nabla u |^2 \, dx,  \ \ \ \text{ (since $\nabla u=  D \Psi_\rho^{-1}\left(\nabla v \circ \Psi_\rho^{-1}\right) $)}\\
           &\leq C' \rho   \int_{\Omega \cap(B_{2\rho}\setminus B_\rho)} | \nabla u |^2 \, dx,  \ \ \ \text{(since $|D\Psi_\rho|$ is bounded)}
               \end{align*}
    where $\nabla v\circ \Psi_\rho^{-1}$ is the gradient of $v$ evaluated at $\Psi_\rho^{-1}$, and $ D \Psi_\rho^{-1}$ is the Jacobian matrix of $\Psi_\rho^{-1}$.  
\end{proof}
\end{proof}
\end{proof}

After proving the above lemma, we are ready to prove Theorem \ref{Gepsilon.theorem}. As mentioned before, this theorem is similar to Theorem 4.2 in \cite{ignatk2}, but here we have an external field term added to the energy. Therefore, in some steps of the following proof, we refer to Theorem 4.2 in \cite{ignatk2}.
\begin{proof}
\begin{enumerate}
\item    To prove the first part of this theorem, let $\left(\phie\right)_{\varepsilon}$ be a sequence/ family in $H^{1}(\Omega; \R)$ such that
\begin{equation*}
    \limsup _{\varepsilon \rightarrow 0} \frac{1}{|\log \varepsilon|} \Ge\left(\phie;\hex\right)<\infty,
\end{equation*}
but
\begin{equation*}
\Ge(\phie;\hex) = \Ge(\phie) - \int_\Omega \hex\cdot e^{i\phie} dx , 
\end{equation*}
where $\Ge(\phie)= \int_{\Omega}|\nabla \phie|^{2} d x+\frac{1}{2 \pi \varepsilon} \int_{\partial \Omega} \sin ^{2}(\phie-g) d \mathcal{H}^{1}$ is the energy functional studied in the original theorem without an external field. 
 Now, since $\left| -\int_{\Omega} \hex\cdot e^{i\phie} \,dx    \right| \leq |\hex| |\Omega|$, we get
 $\limsup _{\varepsilon \rightarrow 0} \frac{-1}{|\log \varepsilon|}\int_{\Omega} \hex\cdot e^{i\phie} \,dx = 0$ . Hence,
\begin{equation*}
    \limsup _{\varepsilon \rightarrow 0} \frac{1}{|\log \varepsilon|} \Ge\left(\phie\right)<\infty , 
\end{equation*}
which is the lower bound assumption in Theorem 4.2 in \cite{ignatk2}, so the conclusion of this part follows (except the first-order lower bound \eqref{lowerbound1} that we have to show), and we get the first-order lower bound 
\begin{equation*}
\liminf _{\varepsilon \rightarrow 0} \frac{1}{|\log \varepsilon|} \Ge\left(\phie\right) \geq \pi \sum_{j=1}^{N}\left|d_{j}\right|  .
\end{equation*}
Hence, 
\begin{align*}
    \liminf _{\varepsilon \rightarrow 0} \frac{1}{|\log \varepsilon|} \Ge\left(\phie;\hex\right) & = \liminf _{\varepsilon \rightarrow 0} \left( \frac{1}{|\log \varepsilon|} \Ge\left(\phie\right) -   \frac{1}{|\log \varepsilon|} \int_{\Omega} \boldsymbol{h}_{ext}\cdot e^{i\phi} \,dx\right)\\
    & = \liminf _{\varepsilon \rightarrow 0}\frac{1}{|\log \varepsilon|} \Ge\left(\phie\right)  \geq \pi \sum_{j=1}^{N}\left|d_{j}\right|.
\end{align*}
\item To prove the second part of this theorem, let $\left(\phie\right)_{\varepsilon}$ be a sequence/ family in $H^{1}(\Omega ; \R)$ satisfying the convergence in part one with the limit $\phi_{0}$ on $\partial \Omega$ as $\varepsilon \rightarrow 0$, and assume
\begin{equation}\label{ghextbounded}
\limsup _{\varepsilon \rightarrow 0}\left(\Ge\left(\phie;\hex\right)-\pi|\log \varepsilon| \sum_{j=1}^{N}\left|d_{j}\right|\right)<\infty.
\end{equation}
Then, since $\limsup _{\varepsilon \rightarrow 0}\left(-\int_{\Omega} \boldsymbol{h}_{ext}\cdot e^{i\phie} \,dx \right)$ is finite, we get 
\begin{equation}\label{gepsilonbounded}
    \limsup _{\varepsilon \rightarrow 0}\left(\Ge\left(\phie\right)-\pi|\log \varepsilon| \sum_{j=1}^{N}\left|d_{j}\right|\right)<\infty , 
\end{equation}
 so the conclusion of part two of Theorem 4.2 in \cite{ignatk2} follows, which means that the conclusion of this part also follows, except the lower bound \eqref{secondlowerbound} that we will prove now.

 To do this, we start by writing the sequence $\left(\phie\right)_{\varepsilon}$ as a sum of two sequences, one of them is harmonic in $\Omega$ and the other one is zero on the boundary $\partial\Omega$, i.e. $\phie=\lambdae+ \thetae $, where $\lambdae\in H^1(\Omega;\R)$ satisfies 
\begin{equation*}
  \begin{cases}
   \Delta\lambdae  = 0  &  \text{in}  \ \Omega, \\ 
 \lambdae=\phie & \text{on} \ \partial\Omega.
      \end{cases}
     \end{equation*} 
Moreover, $\thetae \in H^1(\Omega;\R)$, and $\thetae=0$  on $\partial\Omega$. Then, the first term in $\Ge\left(\phie; \hex\right)$ can be written as
\begin{equation*}
\int_\Omega |\nabla \phie|^2 \,dx =  \int_\Omega |\nabla \lambdae|^2 \,dx + \int_\Omega |\nabla \thetae|^2 \,dx + 2  \int_\Omega \nabla \lambdae \cdot \nabla \thetae \,dx. 
\end{equation*} 
The last term to the RHS above is zero since 
\begin{equation*}
    \int_\Omega \nabla \lambdae \cdot \nabla \thetae = -\int_\Omega \thetae\Delta \lambdae + \int_{\partial\Omega}\frac{\partial\lambdae}{\partial\nu}  \thetae =0
\end{equation*}
(using Green's formula and the assumptions on $\lambdae$ and $\thetae$). Therefore,  
\begin{align*}
    \liminf _{\varepsilon \rightarrow 0}\left(\Ge\left(\phie; \hex\right)-\pi N|\log \varepsilon|\right) & = \liminf _{\varepsilon \rightarrow 0} \Big( \int_\Omega |\nabla \phie|^2 dx  + \frac{1}{2\pi \varepsilon} \int_{\partial \Omega} \sin^2(\phie-g) d\mathcal{H}^1 \\
    & - \int_{\Omega}  \boldsymbol{h}_{ext}\cdot e^{i\phie}dx -\pi N|\log \varepsilon| \Big) \\
    & = \liminf _{\varepsilon \rightarrow 0} \Big(
    \int_\Omega |\nabla \lambdae|^2 \,dx  + \frac{1}{2\pi \varepsilon} \int_{\partial \Omega} \sin^2(\lambdae-g) \, d\mathcal{H}^1 \\ &-\pi N|\log \varepsilon| 
    +\int_\Omega |\nabla \thetae|^2 \,dx - \int_{\Omega}  \boldsymbol{h}_{ext}\cdot  e^{i\left(\lambdae+\thetae \right)} \,dx  \Big) .
\end{align*}
From Theorem 4.2 in \cite{ignatk2} (equation $(48)$), the first three terms in the last equation above satisfy the inequality
\begin{equation*}
     \liminf _{\varepsilon \rightarrow 0} \left(
    \int_\Omega |\nabla \lambdae|^2 \,dx  + \frac{1}{2\pi \varepsilon} \int_{\partial \Omega} \sin^2(\lambdae-g) \, d\mathcal{H}^1  -\pi N|\log \varepsilon| \right) \geq W_{\Omega;0}(\{ (a_j,d_j)\}) + N \gamma_0 , 
\end{equation*}
where $W_{\Omega;0}(\{ (a_j,d_j)\})$ is the renormalized energy when $\hex=0$. To finish the proof and get the lower bound \eqref{secondlowerbound}, we need to show that 
\begin{equation}\label{difference.renormalized.energies.ineq}
    \liminf _{\varepsilon \rightarrow 0} \left( \int_\Omega |\nabla \thetae|^2 \,dx - \int_{\Omega}  \boldsymbol{h}_{ext}\cdot e^{i\left(\lambdae+\thetae \right)} \,dx \right)  \geq W_{\Omega;\hex}(\{ (a_j,d_j)\})-W_{\Omega;0}(\{ (a_j,d_j)\}). 
\end{equation}
%
%
Firstly, we will find a lower bound for the first term above. The sequence $\left( \nabla\thetae\right)_\varepsilon$ is bounded in $L^2(\Omega)$. This is true since 
\begin{align*}
    \|\nabla\thetae \|^2_{L^2(\Omega)} = \int_\Omega |\nabla\thetae|^2 \, dx & =  \Ge(\phie; \hex) - \Ge(\lambdae) + \int_\Omega \hex.e^{i\phie} \, dx, \\
    & \leq \limsup_{\varepsilon\rightarrow 0}\left[  \Ge(\phie; \hex) - \Ge(\lambdae) + \int_\Omega \hex.e^{i\phie} \, dx\right] , 
\end{align*}
by adding and subtracting $\left( -\pi |\log\varepsilon| \sum_{j=1}^{N} |d_j|\right)$, and from \eqref{ghextbounded}, \eqref{gepsilonbounded} and the boundedness of the last term above, we conclude that $\left(\nabla\thetae \right)_\varepsilon$ is bounded in $L^2(\Omega)$. Using Poincaré inequality, we get that $\left(\thetae\right)_\varepsilon$ is bounded in ${L^2(\Omega)}$. This implies that the sequence $\left( \thetae\right)_\varepsilon$ is bounded in $H^1(\Omega)$ (remember that $\left\|\thetae\right\|^2_{H^1(\Omega)} = \left\|\thetae\right\|_{L^2(\Omega)}^2+\left\|\nabla\thetae\right\|_{L^2(\Omega)}^2 $). We know that $H^1(\Omega)$ is a separable Hilbert space, and every bounded sequence in a separable Hilbert space has a weakly convergent subsequence. Therefore, there exists $\theta_0 \in H^1(\Omega)$ such that up to a subsequence $\thetae \rightharpoonup \theta_0$ as $\varepsilon \rightarrow 0$ . Using the weak lower semicontinuity of the Dirichlet integral, we conclude that 
\begin{equation*}
    \liminf_{\varepsilon \rightarrow 0}\left( \int_\Omega |\nabla \thetae|^2 \, dx \right) \geq  \int_\Omega |\nabla \theta_0|^2 \, dx .
\end{equation*}
Let us now find a lower bound for the second term. The sequence $(\phie)_\varepsilon$ is a bounded sequence in $W^{1,q}(\Omega)$, $\forall q \in [1,2)$ as was shown in \cite{ignatk2}. Using Rellich-Kondrachov compactness theorem, we know that $W^{1,q}(\Omega)$ is compactly embedded in $L^p(\Omega)$ ($\forall \ 1\leq p < q^*$, where $\frac{1}{q^*} = \frac{1}{q} - \frac{1}{2}$). Therefore, the bounded sequence $(\phie)_\varepsilon$ converges strongly in $L^p(\Omega)$, so for a further subsequence it converges pointwise almost everywhere
, i.e. there exists $\phi_1 \in L^p(\Omega)$ such that 
\begin{equation*}
    \phie \rightarrow \phi_1 \ \text{  a.e} \ \text{in} \ \Omega . 
    \end{equation*}
 Now, since the exponential function is continuous, we get the convergence for the chosen subsequence of $e^{i\phie}$ to $e^{i\phi_1}$ pointwise almost everywhere. Moreover, $|e^{i\phie}|\leq 1$, so using the dominated convergence theorem, we get
\begin{equation*}
    \lim_{\varepsilon \rightarrow 0} \int_{\Omega}  \hex\cdot e^{i\phie} \,dx =  \int_\Omega  \hex\cdot e^{i\phi_1} \,dx
\end{equation*}
Therefore, we get the lower bound 
\begin{equation}\label{lower}
    \liminf_{\varepsilon \rightarrow 0}\left( \int_\Omega |\nabla \thetae|^2 \, dx  - \int_{\Omega}  \hex\cdot e^{i\phie} \,dx\right) \geq  \int_\Omega |\nabla \theta_0|^2 \, dx -\int_{\Omega}  \hex\cdot e^{i\phi_1} \,dx.
\end{equation}
Going back to the definition of the renormalized energy $\Wad$ (Definition \ref{renormalized.definition}), we know that 
\begin{equation}\label{difference_renormalized_energies}
      \Wad  -   W_{\Omega; 0}(\{(a_j, d_j)\}) =F(\sigmas) = \int_{\Omega} \left( |\nabla \sigmas|^2- \hex\cdot e^{i(\phi_*+\sigmas)}\right) \, dx,
\end{equation}
where $\sigmas$ is any minimizer of the $F(\sigma)= \int_{\Omega} \left( |\nabla \sigma|^2- \hex\cdot e^{i(\phi_*+\sigma)}\right) \, dx$ and $\phi_*$ is the solution of 
\begin{equation*}
  \begin{cases}
  \Delta\phi_*  = 0 \ &\text{in} \ \Omega,\\
    \phi_*  = \phib \  & \text{on} \ \partial\Omega. 
  \end{cases} 
\end{equation*}
Then from equations \eqref{lower} and \eqref{difference_renormalized_energies}, we get the lower bound  
\begin{equation*}
    \liminf _{\varepsilon \rightarrow 0} \left( \int_\Omega |\nabla \thetae|^2 \,dx - \int_{\Omega}  \hex\cdot e^{i\left(\lambdae+\thetae \right)} \,dx \right)  \geq W_{\Omega;\hex}(\{ (a_j,d_j)\})-W_{\Omega;0}(\{ (a_j,d_j)\}).
\end{equation*}
which ends the proof of this part.

\item Let $\phi_0:\partial\Omega \rightarrow \R$ be a function satisfying the assumptions of this part, then from part 3 of Theorem 4.2 in \cite{ignatk2},  $\forall \ \varepsilon > 0$ small, $\exists \phihe^0 \in H^1(\Omega;\R)$ such that $\phihe^0  \rightarrow \phi_0$ in $L^p(\partial\Omega)$, and $\phihe^0  \rightarrow \phi_*$ in $L^p(\Omega)$ for every $p\in[1,\infty)$ where $\phi_*$ is the harmonic extension of $\phi_0$ given in Definition \ref{unperturbed.renormalized.def}, i.e. 
$\Delta\phi_* = 0$ in  $\Omega$, $\phi_* = \phi_0$ on $\partial\Omega$.

Let us construct the sequence 
\begin{equation*}
    \phihe = \phihe^0 + \sigmas,
\end{equation*}
where $\sigmas$ is a minimizer of $F(\sigma) =\int_{\Omega} \left( |\nabla \sigma|^2- \hex\cdot e^{i(\sigma+\phi_*)}\right) \, dx$ and $\sigmas=0$ on $\partial\Omega$. Then 
    \begin{equation*}
        \phihe \rightarrow \phi_0 \ \text{in} \ L^p(\partial\Omega),
    \end{equation*}
    (since $\phihe^0 \rightarrow \phi_0$ in $ L^p(\partial\Omega)$, and $\sigmas=0$ on $\partial\Omega$), and 
\begin{equation*}
    \phihe \rightarrow \sigmas + \phi_* \ \text{in} \ L^p(\Omega).
\end{equation*}

    To prove that $\limsup _{\varepsilon \rightarrow 0} \frac{1}{|\log \varepsilon|} \Ge\left(\phihe; \hex\right)=\pi \sum_{j=1}^{N}\left|d_{j}\right|$, let us firstly recall $\Ge\left(\phihe; \hex\right)$:
\begin{equation*}
\Ge\left(\phihe; \hex\right) =  \int_\Omega |\nabla \phihe |^2 \, dx  + \frac{1}{2\pi \varepsilon}  \int_{\partial\Omega} sin^2(\phihe^0-g) \, d\mathcal{H}^1 - \int_{\Omega} \hex\cdot e^{i \phihe} \, dx.
\end{equation*}
The first term of the above equation can be written as
    \begin{equation*}
        \int_\Omega |\nabla \phihe |^2 \, dx =  \int_\Omega |\nabla (\phihe^0 + \sigmas) |^2 \, dx = \int_\Omega \left( |\nabla\phihe^0|^2 + |\nabla\sigmas|^2 + 2 \nabla\phihe^0\cdot \nabla\sigmas\right)\, dx .
    \end{equation*}
From the proof of Lemma \ref{renormalized_lemma}, we know that $\|\nabla\sigmas\|_{L^2(\Omega)}$ is bounded. 
 Hence, 
\begin{equation*}
\limsup_{\varepsilon \rightarrow 0 } \frac{1}{|\log \varepsilon|}   \int_\Omega |\nabla \sigmas |^2 \, dx = 0.
\end{equation*}
Moreover, the sequence $\phihe^0$ satisfies the assumption of part 2 of Theorem 4.2 in \cite{ignatk2}, so $\nabla \phihe^0 \rightharpoonup \nabla \phi_*$ in $L^q(\Omega; \R^2)$ (where $q \in (1,2]$). Moreover, $\sigmas$ satisfies the Euler-Lagrange equation associated with the functional $F$, so by the regularity theorem we get that $\nabla\sigmas \in L^r(\Omega)$ (where $\frac{1}{r}= 1-\frac{1}{q}$). Therefore, by the definition of weak convergence in $L^q(\Omega)$, we get 
\begin{equation*}
 \int_\Omega \nabla\phihe^0\cdot \nabla\sigmas\, dx \rightarrow \int_\Omega \nabla\phi_*. \nabla \sigmas\, dx.
\end{equation*}
 By Green's Formula and the assumptions $\sigmas=0$ on $\partial \Omega$ and $\phi_*$ is harmonic in $\Omega$, we conclude   $ 2\int_\Omega \nabla\phi_*\cdot \nabla\sigma\, dx = 0$. 
Hence, 
\begin{align*}
    \limsup _{\varepsilon \rightarrow 0} \frac{1}{|\log \varepsilon|} \Ge\left(\phihe; \hex\right) & =\limsup _{\varepsilon \rightarrow 0} \frac{1}{|\log \varepsilon|} \Big( \int_\Omega |\nabla \phihe^0 |^2 \, dx + \frac{1}{2\pi \varepsilon}  \int_{\partial\Omega} sin^2(\phihe^0-g) \, d\mathcal{H}^1 \\
    & - \int_{\Omega} \hex\cdot e^{i \phihe} \, dx\Big)\\
    & =\limsup _{\varepsilon \rightarrow 0} \frac{1}{|\log \varepsilon|} \left( \int_\Omega |\nabla \phihe^0 |^2 \, dx + \frac{1}{2\pi \varepsilon}  \int_{\partial\Omega} sin^2(\phihe^0-g) \, d\mathcal{H}^1\right)\\
     & = \limsup _{\varepsilon \rightarrow 0} \frac{1}{|\log \varepsilon|} \Ge\left(\phihe^0\right)=\pi \sum_{j=1}^{N}\left|d_{j}\right| .
\end{align*}
Note that the term with $\hex$ above vanishes as $\varepsilon$ goes to $0$ (as we did in the proof of part one). 

Finally, to prove the last equality of this part we write $\Ge\left(\phihe; \hex\right)$ as 
\begin{equation*}
\Ge\left(\phihe; \hex\right) = \Ge\left(\phihe^0\right) + \int_\Omega \left(|\nabla \sigmas |^2 - \hex\cdot e^{i\phihe}  + 2 \nabla\phihe^0\cdot \nabla\sigmas\right)\, dx .
\end{equation*}
The last term will vanish as $\varepsilon \to 0$ as shown above, so
\begin{align*}
   \limsup _{\varepsilon \rightarrow 0}\left(\Ge\left(\phihe ; \hex\right)-N \pi \log \frac{1}{\varepsilon}\right)=
     \limsup _{\varepsilon \rightarrow 0}\Big(&\Ge\left(\phihe^0 \right)-N \pi \log \frac{1}{\varepsilon}\\
     & + \int_\Omega \left(|\nabla \sigma |^2 -\hex\cdot e^{i\phihe}\right)\, dx \Big), 
\end{align*}
but from part 3 of Theorem 4.2 in \cite{ignatk2}, we know that 
\begin{equation*}
  \limsup _{\varepsilon \rightarrow 0}\left(\Ge\left(\phihe^0\right)-N \pi \log \frac{1}{\varepsilon}\right)  =  W_{\Omega;0}\left(\left\{\left(a_{j}, d_{j}\right)\right\}\right)+N \gamma_{0}, 
\end{equation*}
so to get the conclusion, we just need to prove that 
\begin{equation*}
     \limsup _{\varepsilon \rightarrow 0}\int_\Omega \left(|\nabla \sigmas |^2 - \hex\cdot e^{i\phihe}\right)\, dx  =  W_{\Omega;\hex}\left(\left\{\left(a_{j}, d_{j}\right)\right\}\right) - W_{\Omega;0}\left(\left\{\left(a_{j}, d_{j}\right)\right\}\right).
\end{equation*}
Using the dominated convergence theorem, we get 
\begin{align*}
     \limsup _{\varepsilon \rightarrow 0}\int_\Omega \left(|\nabla \sigmas |^2 - \hex\cdot e^{i\left(\phihe^0+\sigmas\right)}\right)\, dx & = \int_\Omega \left(|\nabla \sigmas |^2 - \hex\cdot e^{i\left(\phi_*+\sigmas\right)}\right)\, dx ,\\
     &= F(\sigmas),\\
     &=  W_{\Omega; \hex}(\{(a_j, d_j)\}) -   W_{\Omega; 0}(\{(a_j, d_j)\}),
\end{align*}
which ends the proof of this part.

\end{enumerate}
\end{proof}

\section{Gamma-convergence for the reduced energy (Theorem \ref{epsilon.eta.energy})}
Here we prove Theorem \ref{epsilon.eta.energy}, which has some compactness and $\Gamma$-convergence results for the energy $\eee$. 
\begin{proof}
     Assume $u_{\varepsilon} \in H^{1}\left(\Omega ;\R^{2}\right)$ and
\begin{equation}
\limsup _{\varepsilon \rightarrow 0} \frac{1}{|\log \varepsilon|} \eee\left(u_{\varepsilon}; \hex\right)<\infty .
\end{equation}
\begin{enumerate}
 \item If we prove the following claim, then this part follows immediately from Theorem 18 in \cite{ignatk}.
    
    Claim: If
\begin{equation*}
\limsup _{\varepsilon \rightarrow 0} \frac{1}{|\log \varepsilon|} \eee\left(u_{\varepsilon}; \hex\right)<\infty, 
\end{equation*}
then
\begin{equation*}
\limsup _{\varepsilon \rightarrow 0} \frac{1}{|\log \varepsilon|} \eee\left(u_{\varepsilon}\right)<\infty,  
\end{equation*}
where $\eee\left(u_{\varepsilon}\right)$ is the energy with no external field, i.e.
\begin{equation*}
\eee\left(u_{\varepsilon}; \hex\right) = \eee\left(u_{\varepsilon}\right) - \int_{\Omega} \hex.u_\varepsilon \, dx.   
\end{equation*}
\begin{proof}
Assume
\begin{align*}
   \eee(u_\varepsilon;\hex) =& \int_{\Omega}|\nabla u_\varepsilon|^2 \,dx  + \frac{1}{2 \pi \varepsilon} \int_{\partial\Omega}(u_\varepsilon \cdot \nu)^2 d\mathcal{H}^1\\
   &+\frac{1}{\eta^2 } \int_{\Omega}( 1-|u_\varepsilon|^2 )^2\,dx - \int_{\Omega} \hex\cdot u_\varepsilon \,dx,\\
   &\leq C|\log\varepsilon|,
\end{align*}
where $C>0$ is a constant. The first two terms above are non-negative, therefore,  
\begin{equation}\label{last.two.terms}
    \frac{1}{\eta^2 } \int_{\Omega}( 1-|u_\varepsilon|^2 )^2\,dx - \int_{\Omega} \hex\cdot u_\varepsilon \,dx \leq C|\log\varepsilon|.
\end{equation}
Now, we want to prove the following claim to finish the proof:

Claim: If  
\begin{equation*}
    \frac{1}{\eta^2 } \int_{\Omega}( 1-|u_\varepsilon|^2 )^2\,dx - \int_{\Omega} |\hex| |u_\varepsilon| \,dx \leq C|\log\varepsilon|,
\end{equation*}
then $$\int_{\Omega} |u_\varepsilon| \,dx \leq C.$$
To prove this claim, let's state and prove the following lemma:
\begin{lemma}
    Let $0\leq x\in \R$, then there exists a constant $C$, $\varepsilon_0>0$ and $\eta_0>0$ depending on $\hex$ 
    such that for all $\varepsilon>0$ and $\eta>0$ with $\varepsilon<\varepsilon_0$, $\eta<\eta_0$ and $\varepsilon\geq\eta$ 
     the following inequality holds:
    \begin{equation*}
          x \leq  \frac{1}{\eta^2 |\log\varepsilon|}( 1-x^2 )^2 - \frac{|\hex|}{|\log\varepsilon|} x + C.
    \end{equation*}
\end{lemma}
\begin{proof}
    Assume $0\leq x\in \R$, $\varepsilon_0,\, \eta_0 >0$, $0<\varepsilon<\varepsilon_0$ and $0<\eta< \eta_0$ such that $\varepsilon\geq\eta$ (i.e. $|\log\varepsilon| \leq |\log\eta|$). Then $\eta^2|\log\varepsilon| \leq \eta^2|\log\eta|$, and as $\eta$ goes to zero we get that $\eta^2|\log\eta|$ goes to zero. Therefore, $\eta^2|\log\varepsilon|$ will be less than one, and $\frac{1}{\eta^2|\log\varepsilon|}>1$. Now if $x\geq 2$, we get 
     \begin{align*}
     x^2 -1 \geq  \frac{ x^2}{2} \geq 1 &\Rightarrow ( 1-x^2 )^2 \geq \frac{ x^4}{4}, \\
     &\Rightarrow  \frac{1}{\eta^2 |\log\varepsilon|}( 1-x^2 )^2 - \frac{|\hex|}{|\log\varepsilon|} x\geq  x\left( \frac{x^3}{4} - \frac{|\hex|}{|\log\varepsilon|}  \right) >x .
    \end{align*}  
    If $x < 2$, then there is a constant $C$ such that 
    \begin{equation*}
        \frac{1}{\eta^2|\log\varepsilon|}(1-x^2)^2 - \frac{|\hex| }{|\log\varepsilon|} x> -x> x-C.
    \end{equation*}
\end{proof}
Now, using the above lemma and \eqref{last.two.terms}, we get that there exists a constant $C_1$, which may depend on the external field $\hex$, such that
    \begin{align*}
        |u_\varepsilon| & \leq  \frac{1}{\eta^2 |\log\varepsilon|}( 1-|u_\varepsilon|^2 )^2 - \frac{|\hex|}{|\log\varepsilon|} |u_\varepsilon| + C_1,\\
        \Rightarrow   \int_{\Omega}|u_\varepsilon| \, dx & \leq  \int_{\Omega} \frac{1}{\eta^2 |\log\varepsilon|}( 1-|u_\varepsilon|^2 )^2 \,dx - \int_{\Omega} \frac{|\hex|}{|\log\varepsilon|} |u_\varepsilon| \, dx + \int_{\Omega} C_1 \, dx \leq C + \int_{\Omega} C_1 \, dx.
    \end{align*}
    The last term above is bounded since $\Omega$ is a bounded domain, hence, we get that
    \begin{equation*}
\limsup _{\varepsilon \rightarrow 0} \frac{1}{|\log \varepsilon|} \eee\left(u_{\varepsilon}\right)<\infty.  
\end{equation*}
\end{proof} 
%
%
%
\item Assume $(u_\varepsilon)$ satisfies the convergence assumption on the Jacobians as $\varepsilon \rightarrow 0$, then from Theorem 18 in \cite{ignatk} we get:
\begin{equation*}
 \liminf_{\varepsilon\rightarrow 0} \frac{1}{|\log\varepsilon|} \eee(u_\varepsilon) \geq \pi \sum_{j=1}^{N}|d_j|, 
\end{equation*}
but $\liminf_{\varepsilon\rightarrow 0} \frac{1}{|\log\varepsilon|} \int_\Omega \hex.u_\varepsilon = 0$. Therefore, we get the lower bound
\begin{equation*}
    \liminf_{\varepsilon\rightarrow 0} \frac{1}{|\log\varepsilon|} \eee(u_\varepsilon; \hex) \geq \pi \sum_{j=1}^{N}|d_j|. 
\end{equation*}
\item Assume the sharper bound
\begin{equation*}
\limsup _{\varepsilon \rightarrow 0}\left(\eee\left(u_{\varepsilon};\hex\right)-\pi \sum_{j=1}^{N}\left|d_{j}\right||\log \varepsilon|\right)<\infty.
\end{equation*}
Then, we get a sharper bound on $\eee(u_\varepsilon)$:
\begin{equation}
\limsup _{\varepsilon \rightarrow 0}\left(\eee\left(u_{\varepsilon}\right)-\pi \sum_{j=1}^{N}\left|d_{j}\right||\log \varepsilon|\right)<\infty,
\end{equation}
since $\int_\Omega \hex.u_\varepsilon<\infty$. The first part of the third point then follows from Theorem 18 (i.e. $d_j = \pm 1$). 

To finish the proof of this part, we still have to show that the following lower bound holds:
\begin{equation*}
\liminf _{\varepsilon \rightarrow 0}\left(\eee\left(u_{\varepsilon}; \hex\right)-\pi N|\log \varepsilon|\right) \geq W_{\Omega; \hex}\left(\left\{\left(a_{j}, d_{j}\right)\right\}\right)+\gamma_{0} N.
\end{equation*}
The sequence/ family $(u_\varepsilon)$ is a sequence of maps in $H^{1}\left(\Omega ;\R^{2}\right)$ with $\eee(u_\varepsilon;\hex) \leq C|\log\varepsilon|$, which implies $\eee(u_\varepsilon) \leq C|\log\varepsilon|$. Then following the proof of Theorem 1.2 page 48 in \cite{ignatk2}, we can construct a sequence/ family $U_\varepsilon$ such that:
\begin{itemize}
    \item $U_\varepsilon \in H^1(\Omega;\Ss^1)$, 
    \item $\|U_\varepsilon - u_\varepsilon \|_{L^p(\Omega)} \rightarrow 0 $, $\|U_\varepsilon - u_\varepsilon \|_{L^p(\partial\Omega)} \rightarrow 0 $ as $\varepsilon \rightarrow 0$, $\forall p \in [1, \infty)$,
    \item $\eee(U_\varepsilon) \leq\eee(u_\varepsilon) + o_\varepsilon(1)$,
    \item $J(U_\varepsilon) - J(u_\varepsilon) \rightarrow 0 $
in $\left(Lip(\Omega)\right)^*$ as $\varepsilon \rightarrow 0$, where we define the norm $\|A\|_{\left(Lip(\Omega)\right)^*}$ for $A\in \left(W^{1,\infty}(\Omega)\right)^*$ (the dual space of $W^{1,\infty}(\Omega)$) as follows:
\begin{equation*}
   \|A\|_{\left(Lip(\Omega)\right)^*} = \sup \{ \langle A,\xi\rangle: \xi \in W^{1,\infty}(\Omega), |\nabla\xi|\leq 1 \}. 
\end{equation*}
\end{itemize}
That is, these maps $u_\varepsilon:\Omega\to\R^2$ can be approximated by $\Ss^1$ valued maps $U_\varepsilon:\Omega\to\Ss^1$ such that the energies of $u_\varepsilon$ and $U_\varepsilon$ are close, and their global Jacobians are close too. Moreover, the distance between $u_\varepsilon$ and $U_\varepsilon$ in $L^p(\Omega)$ and $L^p(\partial\Omega)$ goes to zero as $\varepsilon$ goes to zero for all $p\in[1,\infty)$. 

The idea to prove this starts by proving it for the unit disk $\Omega=B_1$, and to do so, a polar grid $\mathcal{R}$ is constructed in $B_1$, which divides the disk into small cells $\mathcal{C}\subset\mathcal{R}$. Then, an approximating $\Ss^1$-valued map $\hat{U}_\varepsilon$ for $u_\varepsilon$ is introduced inside the interior of the grid $\mathcal{R}$ by solving a minimization problem for the Ginzburg-Landau energy on each cell of the grid (and the minimizer is denoted by $\omega=\omega_\varepsilon\in H^1(int(\mathcal{C}), \R^2)$)  and then defining the minimizer in the whole interior of the grid. Then showing that as $\varepsilon$ goes to zero, the norm of this minimizer $\omega$ approaches 1 (and the norm of $u_\varepsilon$ is more than or equals $1/2$ on the grid) (see Corollary 4 in \cite{ignatk.lamy}) and defining an $\Ss^1$-valued map as  $\hat{U}_\varepsilon:=\omega/|\omega|$ in the interior of $\mathcal{R}$. Next, using $\hat{U}_\varepsilon$ and a transformation relating the disk $B$ and the interior of the grid $\mathcal{R}$, the approximating $\Ss^1$-valued map $U_\varepsilon$ of $u_\varepsilon$ is defined in $B_1$. Finally, this approximating $\Ss^1$-valued map is defined for the simply connected $C^{1,1}$ domain $\Omega$ using the existence of a conformal map transforming $\Omega$ and $\partial\Omega$ to $B_1$ and $\partial B_1$, respectively. for more details, see the proof of Theorem 3.1 in \cite{ignatk2}.

From the second point above,  $\|U_\varepsilon - u_\varepsilon \|_{L^p(\Omega)} \rightarrow 0$  as  $\varepsilon \rightarrow 0$, so $\lim_{ \varepsilon \rightarrow 0}\int_\Omega \left( U_\varepsilon-u_\varepsilon\right) \, dx = 0 $, which implies 
\begin{equation}\label{uepsilon}
 \lim_{ \varepsilon \rightarrow 0}\left( -\int_\Omega\hex\cdot U_\varepsilon \, dx\right)= \lim_{ \varepsilon \rightarrow 0}\left( -\int_\Omega\hex\cdot u_\varepsilon \, dx\right), 
\end{equation}
for a fixed external field $\hex$. Now, using Lemma \ref{lifting} above, there exists liftings $\phie \in H^1(\Omega;\R)$ such that $U_\varepsilon = e^{i\phie}$ and $\eee(U_\varepsilon) = \Ge(\phie)$. Then, from the third assumption on the sequence $U_\varepsilon$ we get
\begin{equation*}
 \Ge(\phie) = \eee(U_\varepsilon)\leq \eee(u_\varepsilon) +o_\varepsilon(1), 
\end{equation*}
and from \eqref{uepsilon}, 
\begin{align*}
 \Ge(\phie) - \lim_{ \varepsilon \rightarrow 0}\left( -\int_\Omega\hex\cdot U_\varepsilon \, dx\right) &-\pi N |\log\varepsilon|\\
 &\leq \eee(u_\varepsilon) + \lim_{ \varepsilon \rightarrow 0}\left( -\int_\Omega\hex\cdot u_\varepsilon \, dx\right) -\pi N |\log\varepsilon|+o_\varepsilon(1),
\end{align*}
so 
\begin{equation*}
    \liminf_{\varepsilon \rightarrow 0} \left(  \Ge(\phie;\hex)  -\pi N |\log\varepsilon|\right) \leq \liminf_{\varepsilon \rightarrow 0} \left( \eee(u_\varepsilon;\hex) -\pi N |\log\varepsilon|+o_\varepsilon(1)\right).
\end{equation*}
From Theorem \ref{Gepsilon.theorem}, we get the second-order lower bound
\begin{align*} 
 W_{\Omega; \hex}\left(\left\{\left(a_{j}, d_{j}\right)\right\}\right)+\gamma_{0} N\leq\liminf_{\varepsilon \rightarrow 0} \big(  \Ge(\phie;\hex) & -\pi N |\log\varepsilon|\big)\\
 &\leq \liminf_{\varepsilon \rightarrow 0} \left( \eee(u_\varepsilon;\hex) -\pi N |\log\varepsilon|\right).
\end{align*}
Note that we can use this theorem since $\left( \phie \right)$ satisfies the convergence assumption in Theorem \ref{Gepsilon.theorem} -see the proof of Theorem 1.2 page 49 in \cite{ignatk2} for more details- and since 
\begin{equation*}
 \limsup_{\varepsilon\rightarrow0}\left( \Ge(\phie; \hex) - \pi |\log\varepsilon| \sum_{j=1}^N |d_j| \right) \leq \limsup_{\varepsilon\rightarrow0}\left( \eee(u_\varepsilon; \hex) - \pi |\log\varepsilon| \sum_{j=1}^N |d_j| \right) < \infty.
\end{equation*}
\item The proofs of this part and the next two parts follow immediately from Theorem 18 in \cite{ignatk}.
\item[7.] 
The proof of this point follows from Theorem 18, except for the last equation that follows if we set $u_\varepsilon= e^{i\psi_\varepsilon}$ for $\psi_\varepsilon\in H^1(\Omega;\R)$ (such a lifting exists, see Lemma \ref{lifting}) and using the last equation in Theorem \ref{Gepsilon.theorem}, which we can use since $\eee(u_\varepsilon) = \Ge(\psi_\varepsilon)$.

\end{enumerate}

\end{proof}
\section{Gamma-convergence for the full 3D energy (Theorem \ref{rescaled.energy.compact})}
We start this section by proving Theorem \ref{rescaled+reduced} that connects the 3D energy $E_\alpha$ with the 2D energy $\eeeb$ introduced before. 
\begin{proof}(of Theorem \ref{rescaled+reduced})
Consider a family of magnetizations $\{\bfm_\alpha:\bfOm_\alpha\to \Ss^2\}_{\alpha\to 0}$ such that $\limsup_{\alpha\to 0} E_\alpha(\bfm_\alpha;\hext)<\infty$.
Then we get \begin{equation*}
    \limsup_{\alpha\to 0} E_\alpha(\bfm_\alpha)<\infty,
\end{equation*}
since \begin{equation*}\label{average.energy}
     E_\alpha(\bfm_\alpha;\hext) = E_\alpha(\bfm_\alpha) - \frac{1}{\alpha|\log\varepsilon|}\int_{\bfOm_\alpha} \hext.\bfm_\alpha \, dx = E_\alpha(\bfm_\alpha) - \frac{1}{|\log\varepsilon|}\int_{\bfOm} \hext.\bm \, dx ,
\end{equation*}
the last term above is bounded since $|\bm|\leq 1$ and $\varepsilon$ is very small($\varepsilon \rightarrow 0$), so we get the boundedness of $ \limsup_{\alpha\to 0} E_\alpha(\bfm_\alpha)$. Now, from Theorem 1 in \cite{ignatk}, we know that  
\begin{align*}
    E_\alpha(\bfm_\alpha) &\ge \eeeb(\bm)-o(1) \quad \textrm{as} \quad \alpha\to 0,\\
    \Rightarrow \ E_\alpha(\bfm_\alpha)& - \frac{1}{\alpha|\log\varepsilon|}\int_{\bfOm_\alpha} \hext.\bfm_\alpha \, dx  \ge \eeeb(\bm) - \frac{1}{\alpha|\log\varepsilon|}\int_{\bfOm_\alpha} \hext.\bfm_\alpha \, dx  -o(1) ,\\
    \Rightarrow \ E_\alpha(\bfm_\alpha;&\hext)\ge \eeeb(\bm;\hext)-o(1) \quad \textrm{as} \quad \alpha\to 0. 
\end{align*}
The other results of the theorem can be proven similarly.
\end{proof}
Now, we are ready to prove the main theorem of this paper.
\begin{proof}(Of Theorem \ref{rescaled.energy.compact})
Assume  $\bfm_\alpha\in H^1(\bfOm_\alpha; \Ss^2)$ is a sequence of magnetizations such that
$$\limsup_{\alpha\to 0} E_\alpha(\bfm_\alpha;\hext)<\infty.$$
Then \begin{equation*}
    \limsup_{\alpha\to 0} E_\alpha(\bfm_\alpha)<\infty,
\end{equation*}
 as we showed in the proof of Theorem \ref{rescaled+reduced} above.
 \begin{enumerate}
     \item The proof of this part follows immediately from Theorem \ref{theorem9ignatk}.
     \item From Theorem \ref{rescaled+reduced}, equation\eqref{lowerbound2}, and part 2 of Theorem \ref{epsilon.eta.energy}, we get
     \begin{equation}
         \liminf_{\alpha\to 0} E_\alpha(\bfm_\alpha;\hext) \ge \liminf_{\alpha\to 0}\eeeb(\bm;\hext) \ge \liminf_{\alpha\to 0}\frac{1}{|\log\varepsilon|} \eee(\bmp;\hext) \ge \pi\sum_{j=1}^N |d_j|.
     \end{equation}
     \item To prove this part, assume we are working in the narrower regime \eqref{narrower.regime}, and assume 
     \begin{equation}\label{first}
         \limsup_{\alpha\to 0} |\log \varepsilon| \bigg(E_\alpha(\bfm_\alpha;\hext) - \pi\sum_{j=1}^N |d_j|\bigg)\leq C,
     \end{equation}
     then from inequality \eqref{narrower.inequality}, we get  
     \begin{equation}\label{second}
       \limsup_{\alpha\to 0} |\log \varepsilon| \bigg(\eeeb(\bm;\hext) - \pi\sum_{j=1}^N |d_j|\bigg) \leq \limsup_{\alpha\to 0} |\log \varepsilon| \bigg(E_\alpha(\bfm_\alpha;\hext) - \pi\sum_{j=1}^N |d_j|\bigg),
     \end{equation}
     and from \eqref{lowerbound2}, we get
     \begin{equation}\label{third}
       \limsup_{\alpha\to 0} \bigg(\eee(\bm;\hext) - \pi\sum_{j=1}^N |d_j| |\log \varepsilon|\bigg) \leq \limsup_{\alpha\to 0} |\log \varepsilon| \bigg(\eeeb(\bm;\hext) - \pi\sum_{j=1}^N |d_j|\bigg).
     \end{equation}
     From \eqref{first}, \eqref{second}, and \eqref{third}, we get \eqref{sharper.bound}. Therefore, from part 3 of Theorem \ref{epsilon.eta.energy}, we conclude that the multiplicities $d_j=\pm 1$ for $1\leq j\leq N$. Moreover, from the inequality in part 3 of Theorem \ref{epsilon.eta.energy}, inequality \eqref{second}, and \eqref{third}, we get the lower bound
     \begin{equation*}
         \liminf_{\alpha\to 0} |\log \varepsilon| \bigl(E_\alpha(\bfm_\alpha;\hext) - \pi N \bigr) 
\ge W_{\Omega;\hext}(\{(a_j,d_j)\}) + {\gamma_0}N.
     \end{equation*}
     \item Assume we are working on the narrower regime \eqref{narrower.regime} and 
     \begin{equation*}
         \limsup_{\alpha\to 0} |\log \varepsilon| \bigg(E_\alpha(\bfm_\alpha;\hext) - \pi\sum_{j=1}^N |d_j|\bigg)\leq C. 
     \end{equation*}
     Then 
     \begin{align*}
         \limsup_{\alpha\to 0} |\log \varepsilon| & \bigg(E_\alpha(\bfm_\alpha) - \frac{1}{|\log\varepsilon|}\int_{\bfOm_\alpha} \hext.\bfm_\alpha \, dx - \pi\sum_{j=1}^N |d_j|\bigg)\\
         &=\limsup_{\alpha\to 0} |\log \varepsilon| \bigg(E_\alpha(\bfm_\alpha) - \pi\sum_{j=1}^N |d_j|\bigg)  + \limsup_{\alpha\to 0}\left(-\int_{\bfOm} \hext\cdot\bm \, dx\right)\leq C,
         \end{align*}
         but the last term above is bounded, so we get 
         \begin{equation*}
             \limsup_{\alpha\to 0} |\log \varepsilon| \left(E_\alpha(\bfm_\alpha) - \pi\sum_{j=1}^N |d_j|\right) \leq C,
         \end{equation*}
         which is the assumption of parts 3 and 4 of Theorem \ref{theorem9ignatk}. Therefore, the conclusion of this part follows immediately from part 4 of Theorem \ref{theorem9ignatk}. 
 \end{enumerate}
    
\end{proof}
\section{Numerical simulations of the magnetization behaviour}

In this section, we illustrate the behaviour of the magnetization vector field in different domains with different values of the external field applied to them. The illustrations show the location of the vortices when the renormalized energy is minimal and the direction of the magnetization vector field inside and at the boundary of each domain. 

Based on the work done in this paper, and in \cite{ignatk, appliedfield}, we consider having only two boundary vortices $a_1\neq a_2$ on $\partial\Omega$ with multiplicities $d_1=d_2=1$. Moreover, we consider the renormalized energy
\begin{equation*}
    W_{\Omega;\hex}(a)= W_{\Omega;0}(a) +V_{\Omega;\hex}(a),
\end{equation*}
where $(a)$ here is a shorthand notation for the two vortices $a_1$ and $a_2$ and their multiplicities. The energy $W_{\Omega,0}(a)$ is the unperturbed renormalized energy (i.e. the renormalized energy when $\hex=0$) and $V_{\Omega;\hex}(a)$ is the perturbation of the energy when an external field is applied. It is given by
\begin{equation*}
 V_{\Omega;\hex}(a) = \min_{\theta\in H^1_0(\Omega)} G_{\Omega;\hex}(a;\theta),    
\end{equation*}
where 
\begin{equation*}
G_{\Omega;\hex}(a;\theta)=\int_\Omega\frac{1}{2} |\nabla\theta|^2 - \hex\cdot(e^{i\theta} M(x;a)) \, dx.
\end{equation*}
As mentioned in \cite{appliedfield}, for $\hex\in\R^2$ sufficiently small, the functional $G_{\Omega;\hex}(a;\theta)$ is strictly convex on $H^1_0(\Omega)$, so there is a unique minimizer $\theta\in H^1_0(\Omega)$. This $\theta$ satisfies the Euler-Lagrange equation associated with the functional $G_{\Omega;\hex}$ given by 
\begin{equation*}
    \begin{cases}
    -\Delta\theta  = i\hex\cdot(e^{i\theta} M(x;a))  \ \ \text{in}  \ \ \Omega, \\
    \theta=0 \ \ \text{on} \ \ \partial\Omega,
    \end{cases}
\end{equation*}
where $M(\cdot;a)$ is the canonical harmonic map associated to $\{(a_1,d_1), (a_2,d_2)\}$ (i.e. an $\Ss^1$-valued smooth harmonic map $M=e^{i\psi}$ -with $\Delta\psi=0$ in $\Omega$- which is tangent to the boundary of the domain $\partial\Omega$ except at $N$ boundary points $a_i\in \partial\Omega$, i=1,2,...,N, where $M$ winds according to the multiplicities $d_i$. In our case, we assume that $N$ equals two). Theorem 4 in \cite{ignatk} gives the form of the canonical harmonic maps for the unit disk $B_1$ and for domains $\Omega\subset \R^2$ where a $C^1$ conformal diffeomorphism $\Phi$ is defined from the closure of the unit disk $\Bar{B}_1$ to the closure of the domain $\Bar{\Omega}$. In this paper, we consider the case of the unit disk and the case of an oval-shaped domain. 

We also consider the magnetization vector $m$ to be given by 
\begin{equation*}
    m(x;a) = e^{i\theta} M(x;a), \ \ \forall x \in B_1.
\end{equation*}
Therefore, to find the magnetization vector, we must first find $\theta$ by solving the above PDE. To do so, we use Banach fixed point theorem (see Chapter 9 in \cite{evans} for more details) and we consider the non-linear mapping $A:H^1_0(\Omega) \rightarrow H^1_0(\Omega)$ which takes $\theta \in H^1_0(\Omega)$ to $-(\Delta)^{-1} (i\hex\cdot e^{i\theta} M(x;a))$. We fix a point $\theta_0 =0 \in  H^1_0(\Omega)$ and iteratively define $\theta_{n+1} =A[\theta_n]$ for n=0,1,2,..., i.e. $\theta_{n+1}$ satisfies 
\begin{equation*}
    \begin{cases}
    -\Delta\theta_{n+1}  = i\hex\cdot(e^{i\theta_n} M(x;a))  \ \ \text{in}  \ \ \Omega, \\
    \theta_{n+1}=0 \ \ \text{on} \ \ \partial\Omega.
    \end{cases}
\end{equation*}

In our Matlab code, we use the Partial Differential Equation Toolbox to solve the above PDE up to $\theta_4$ (i.e. for n=0,1,2, and 3) on the unit disk $\Omega=B_1$. This toolbox solves general partial differential equations using finite element analysis by meshing the unit disk into triangles, defining the boundary conditions on these smaller elements and solving the problem. As mentioned before, we work mainly on the unit disk and on an oval domain, but our code works for domains $\Omega\subset \R^2$ where a $C^1$ conformal diffeomorphism $\Phi$ from $\Bar{B}_1$ to $\Bar{\Omega}$ exists.    

\subsection{In the unit disk}
Consider the domain to be the unit disk, i.e. $\Omega=B_1$. Let $a_1=e^{is_1}$ and $a_2=e^{is_2}$ be the two boundary vortices on $\partial B_1$, where $s_1$ and $s_2$ are in $[0,2\pi)$ and let the vector $s=(s_1, s_2)$. The canonical harmonic map associated with these vortices is given by (see Theorem 4 in \cite{ignatk})
\begin{equation}\label{canonical.unitdisk}
   M(x;a)=\frac{(x-a_1)(x-a_2)|a_1-a_2|}{|x-a_1||x-a_2|(a_1-a_2)}, \ \ \forall x\in B_1,
\end{equation}
and the unperturbed renormalized energy is (see Theorem 6 in \cite{ignatk})
\begin{equation*}
   W_{\Omega;0}(a)= -\pi \log|a_1-a_2|.
\end{equation*}
Our aim is to find the location of the boundary vortices minimizing the renormalized energy, i.e. to find $s=(s_1, s_2)$ where $W_{\Omega;\hex}$ is minimum. Therefore, in our Matlab code, we use the "fminsearch" function, which looks for a local minimum of the renormalized energy. It searches for the vector $s=(s_1, s_2)$ and it starts from an initial expected point $s_0=[s_{01} \ \ s_{02}]$ we provide. Then, we plot the magnetization vector field $  m(x;a) = e^{i\theta} M(x;a)$ in the unit disk, where $\theta$ is the solution of the above PDE when the boundary vortices minimize the renormalized energy.

When no external field is applied to the unit disk, we expect the boundary vortices to be diametrically opposite for the renormalized energy to be minimum. This is true by looking at the form of the unperturbed renormalized energy and noticing that $\log|a_1-a_2|$ is maximum when the distance between the vortices is maximum.  Figure \ref{disk_m_1} shows an example of the magnetization vector field when $\hex=0$, and we see that the two boundary vortices represented by the two dots are in their expected locations. Note that the location of the vortices minimizing the renormalized energy is not unique as any two diametrically opposite vortices will minimize the renormalized energy. This is shown in Figure \ref{disk_W_1}, where the energy is lowest when $|s_1-s_2|=\pi$.

When the external field is nonzero, we expect the vortices to align in the direction of the field $\hex$. For example, when $\hex=(-0.01, 0)$, we expect the vortices to be diametrically opposite to each other in the direction of $\hex$, which is the outcome of our code as shown in Figure \ref{disk_m_2} Moreover, Figure \ref{disk_W_2} shows that the renormalized energy is minimum when $s_1=0$ and $s_2=\pm \pi$ as expected. 
\begin{figure}

\centering

    \subfloat[]{\label{disk_m_1}
    \centering
    
 \includegraphics[width=0.45\linewidth]{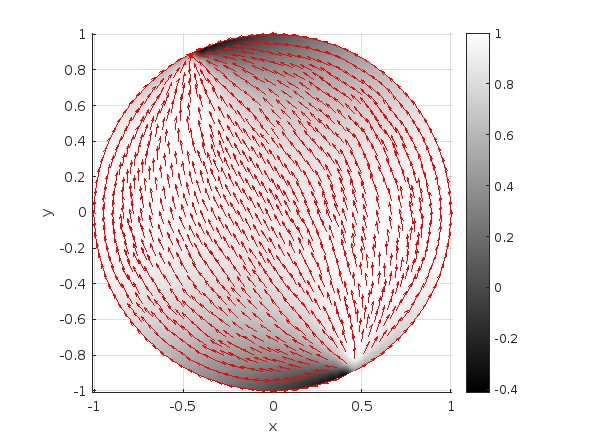}}
   \hfill
   \subfloat[]{\label{disk_W_1}
\centering
     \includegraphics[width=0.45\linewidth]{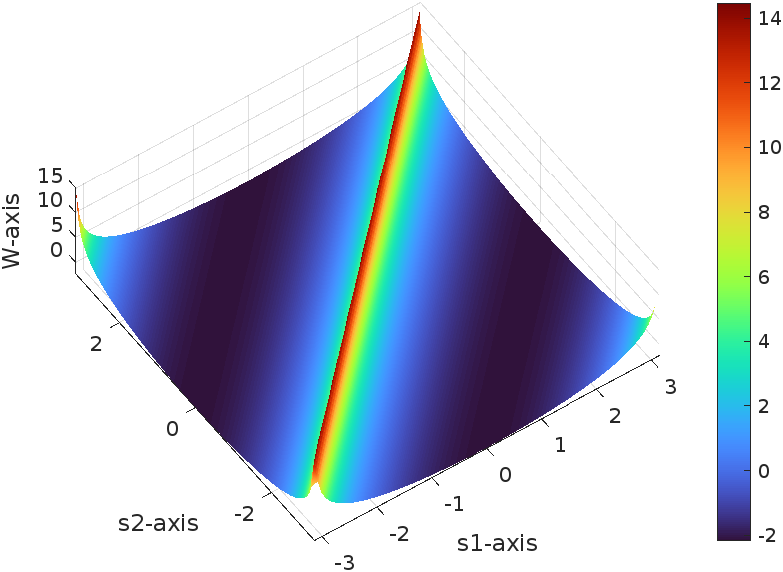}
     }    
     \hfill
       \subfloat[]{\label{disk_m_2}
    \centering
    
 \includegraphics[width=0.45\linewidth]{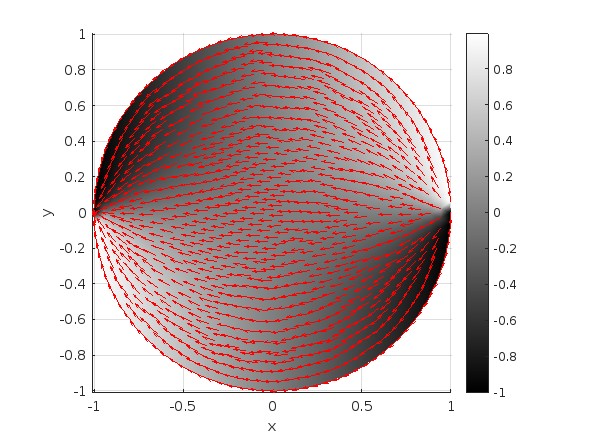}}
 \hfill
  \subfloat[]{\label{disk_W_2}
    \centering
    
 \includegraphics[width=0.5\linewidth]{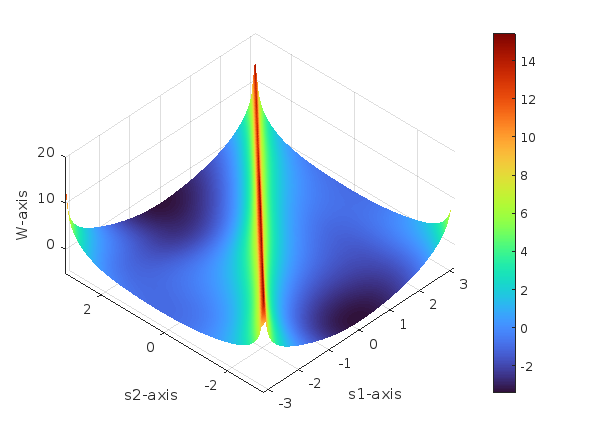}}
 
\caption{Numerical plot of (left) the magnetization vector field $m$ in grayscale with directional arrows and (right)  the renormalized energy versus the location of the boundary vortices (i.e. versus $s_1$ and $s_2$)  when applying an external field $\hex=(0,0)$ for (a) and (b) and $\hex=(-0.01,0)$} for (c) and (d). Notice that the renormalized energy is infinite on the diagonal $s_1=s_2$
\label{h=0.01}
\end{figure}
\subsection{In an oval domain}
Now consider the domain $\Omega\subset \R^2$ to have an oval shape with a 
smooth conformal diffeomorhpism $\Phi:\Bar{B}_1 \to\Bar{\Omega}$ given by 
\begin{equation*}
    \Phi(z)= \frac{z}{1-0.2 z^2} , \ \ z\in \Bar{B}_1,
\end{equation*}
with inverse $\Psi$, where $\Psi(w) = \frac{-1+\sqrt{1+0.8w^2}}{0.4w}$ for  $w \in\Bar{\Omega}$. Then, the map given by 
\begin{equation*}
M_*(w;\Phi(a)) = M(\Psi(w),a) \frac{\Phi'(\Psi(w))}{|\Phi'(\Psi(w))|}, \ \  \forall w \in \Omega,
\end{equation*}
is the canonical harmonic map associated to the boundary vortices $\{(\Phi(a_1),d_1), (\Phi(a_2),d_2)\}$ on $\partial\Omega$ where $a_1$ and $a_2$ are the boundary vortices on $\partial B_1$ mentioned before, and $M(\cdot;a)$ is the canonical harmonic map associated to  $\{(a_1,d_1), (a_2,d_2)\}$ given by equation \eqref{canonical.unitdisk} (see Theorem 4 in \cite{ignatk}). Let $b_1=\Phi(a_1)$ and $b_2=\Phi(a_2)$ be two boundary vortices on the boundary of the oval-shaped domain (on $\partial\Omega$) with $d_1=d_2=1$. Then, the unperturbed renormalized energy is (see Remark 8 in \cite{ignatk})
\begin{equation*}
  W_{\Omega;0}(a) = -\pi \log|a_1-a_2|  +\frac{1}{2} \int_{\partial B_1} \varkappa(\Phi(z)) |\Phi'(z)| \left(\log|z-a_1|+\log|z-a_2|+\log|\Phi'(z)| \right) \, d\mathcal{H}^1(z).
\end{equation*}
Our Matlab code looks for the vector $s=(s_1, s_2)$ where the renormalized energy of this oval-shaped domain is minimum. The outcome of the code is a numerical plot of the magnetization vector field 
\begin{equation*}
    m_*(w;\Phi(a))=e^{i\theta} M_*(w;\Phi(a)), \ \ \forall w \in \Omega,
\end{equation*}
where $\theta$ is the solution of the above PDE in the unit disk (see the previous section) when the renormalized energy in this oval-shaped domain is minimum. 
\begin{figure}

\centering

    \subfloat[]{\label{ellipse_m_1}
    \centering
    
 \includegraphics[width=0.45\linewidth]{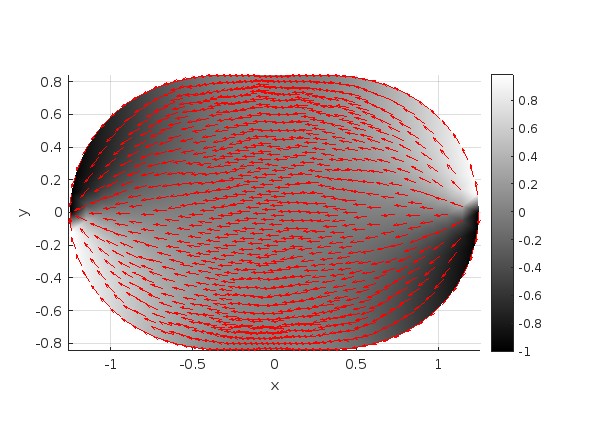}}
   \hfill
   \subfloat[]{\label{ellipse_m_2}
\centering
     \includegraphics[width=0.45\linewidth]{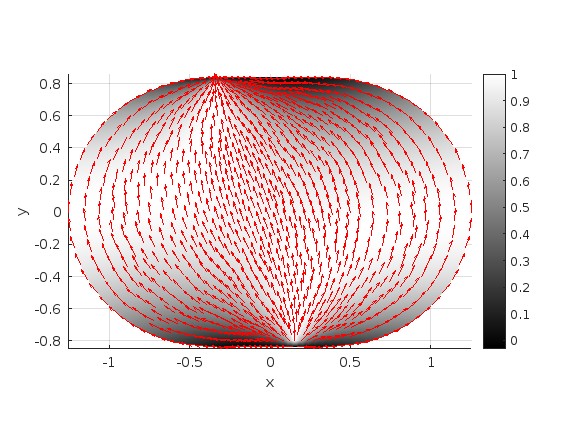}
     }    
     \hfill
       \subfloat[]{\label{ellipse_m_3}
       \centering
     \includegraphics[width=0.45\linewidth]{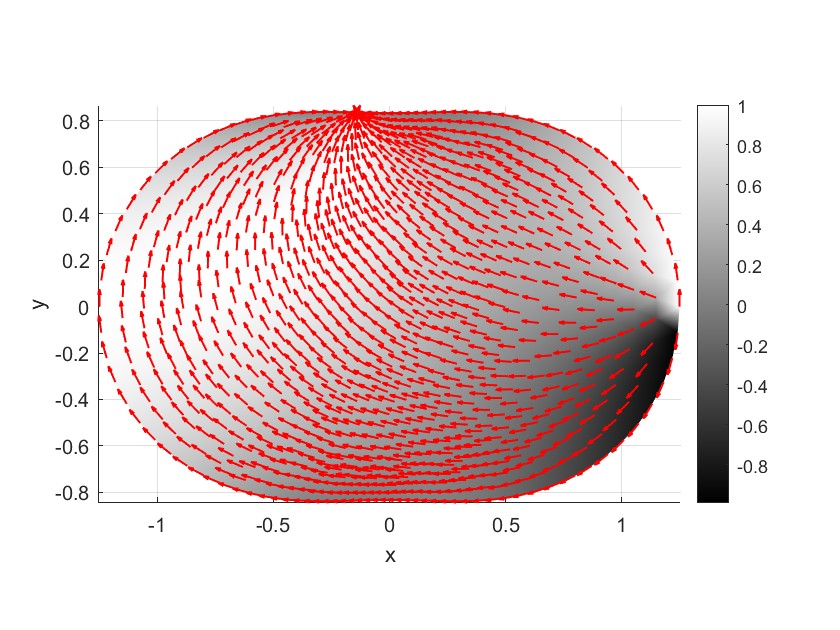}
     }
 \hfill
  \subfloat[]{\label{ellipse_m_4}
    \centering
     \includegraphics[width=0.45\linewidth]{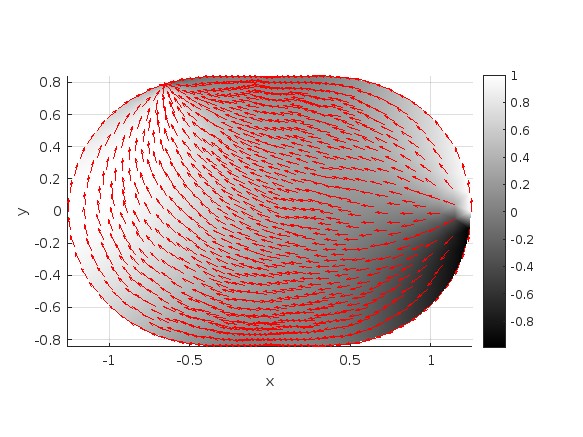}
     }
 
\caption{Numerical plot of the magnetization vector field $m$ in grayscale with directional arrows when applying an external field $\hex=(0,0)$ in (a), $\hex=(0,1)$ in (b),  $\hex=(0,0.01)$ in (c) and $\hex=(-0.01/\sqrt{2},0.01/\sqrt{2})$ in (d)}
\end{figure}

As discussed in \cite{ignatk}, for the renormalized energy to be minimal in the case of no external applied field, the boundary vortices have to be the furthest apart and at the same time the curvature $\varkappa$ at these vortices has to be the highest (there is a nontrivial competition between these two effects). Figure \ref{ellipse_m_1} shows the magnetization behavior when $\hex=0$, where the two boundary vortices appear opposite each other and furthest apart ($s_1=0$ and $s_2=\pm \pi$).

When an external field is applied, we expect the boundary vortices to try to align in the direction of the external field. However, the previously mentioned effects still affect the location of the vortices. For example, when $\hex=(0,1)$, the renormalized energy is minimal when the two vortices are in the locations shown in Figure \ref{ellipse_m_2}. We see that the magnetization vector field between these two vortices tries to align in the direction of $\hex$, but they are not totally aligned in the direction of $\hex$ due to the other effects mentioned before. Suppose we apply a weaker external field, for example $\hex=(0,0.01)$ as in Figure \ref{ellipse_m_3}. In that case, we notice that the magnetization vector field in the domain tries to align in the direction of the external field. Still, comparing with the previous case in Figure \ref{ellipse_m_2}, we see that strengthening the external field cause the magnetization to align more in the direction of the field. Suppose we fix the value of the external field and change the direction, we see that the magnetization field tries to align in the direction of $\hex$ as we notice from Figures \ref{ellipse_m_3} and \ref{ellipse_m_4} where $\hex=(0,0.01)$ and $\hex=(-0.01/\sqrt{2},0.01/\sqrt{2})$, respectively.

\printbibliography 

@BOOK{nanomag,
	title = {Principles of nanomagnetism},
	publisher = {Springer-Verlag Berlin Heidelberg},
	author = {Guimar\~{a}es, A.},
	year = {2009},
	edition = {},}

@BOOK{hubert,
	title = {Magnetic domains},
	publisher = {Springer-Verlag Berlin Heidelberg},
	author = { Hubert, A. and
Schäfer, R. },
	year = {2009},
	edition = {},}

@BOOK{aharoni,
	title = {Intoduction to the theory of ferromagnetism},
	publisher = {Oxford University Press},
	author = {Aharoni, A.},
	year = {2001},
	edition = {2},}

@inbook{recent, 
title = {Recent analytical developments in micromagnetics},
author = {DeSimone, A. and  Kohn, R.  and  M\"{u}ller, S. and Otto, F.},
year = {2006},
pages = {269-381},
editor = {G. Bertotti and I. Mayergoyz},
booktitle = {The science of hysteresis II},
publisher = {Elsevier},
}

@ARTICLE{ignatk,
	author = { Ignat, R.  and Kurzke, M.},
	title = {An effective model for boundary vortices in thin-film
micromagnetics},
	journal = {Mathematical Models and Methods in Applied Sciences},
	year = {2023},
	volume = {},
	pages = {},
	number = {},
	month = {},
	publisher = {World Scientific Publishing Company}
}

@ARTICLE{ignatk2,
	author = {Ignat, R. and Kurzke, M.},
	title = {Global {J}acobian and {$\Gamma$}-convergence in a two-dimensional {G}inzburg-{L}andau model for boundary vortices},
	journal = {Journal of
Functional Analysis},
	year = {2021},
	volume = {280},
	pages = {},
	number = {},
	month = {},
	publisher = {}
}

@ARTICLE{ignatk.lamy,
	author = {Ignat, R. and Kurzke, M. and Lamy, X.},
	title = {Global uniform estimate for the modulus of 2D
{G}inzburg-{L}andau vortexless solutions with asymptotically inﬁnite boundary energy},
	journal = {SIAM Journal on Mathematical Analysis},
	year = {2020},
	volume = {52},
	pages = {},
	number = {},
	month = {},
	publisher = {}
}

@BOOK{Brezis,
	title = {Ginzburg-{L}andau vortices},
	publisher = {Progress in Nonlinear Differential Equations and Their Applications, Birkhauser Boston Inc., Boston,
MA},
author = {Bethuel, F. and Brezis, H. and Helein, F. },
	year = {1994},
	edition = {},
}

@BOOK{rindler,
	title = {Calculus of Variations},
	publisher = {Springer International Publishing AG},
	author = {Rindler, F.},
	year = {2018},
	edition = {},
}

@phdthesis{marco,
  title        = {Singularities in thin magnetic films},
  author       = {Baffetti, M.},
  year         = 2021,
  month        = {},
  address      = {},
  note         = {},
  school       = {University of Nottingham},
  type         = {PhD thesis}
}

@ARTICLE{gioia,
	author = {Gioia, G.  and James, R.},
	title = {Micromagnetics of very thin films},
	journal = {Proceedings: Mathematical, Physical and
Engineering Sciences},
	year = {1997},
	volume = {453},
	pages = {213-223},
	number = {1956},
	month = {},
	publisher = {Royal Society}
}

@ARTICLE{kohn,
	author = { Kohn, R. and Slastikov, V.},
	title = {Another thin-film limit of micromagnetics},
	journal = {Archive for Rational Mechanics and Analysis},
	year = {2005},
	volume = {178},
	pages = {227–245},
	number = {},
	month = {},
	publisher = {}
}

@ARTICLE{kurzke,
	author = {Kurzke, M.},
	title = {Boundary vortices in thin magnetic films},
	journal = {Calculus of Variations},
	year = {2006},
	volume = {26},
	pages = {1–28},
	number = {},
	month = {},
	publisher = {}
}

@ARTICLE{kurzke1,
	author = {Kurzke, M.},
	title = {A nonlocal singular perturbation problem with periodic
well potential},
	journal = {ESAIM: Control, Optimisation and Calculus of Variations},
	year = {2006},
	volume = {12},
	pages = {52 - 63},
	number = {1},
	month = {},
	publisher = {}
}

@ARTICLE{gradientflow,
	author = {Kurzke, M.},
	title = {The gradient flow motion of boundary vortices},
	journal = {Annals of the Institut Henri Poincaré C, Nonlinear Analysis},
	year = {2007},
	volume = {24},
	pages = {91-112},
	number = {1},
	month = {},
	publisher = {}
}

@ARTICLE{Kreisbeck,
	author = {Kreisbeck, C.},
	title = {Another approach to the thin-film Gamma-limit of the micromagnetic free energy in the regime of small samples},
	journal = {Quarterly of Applied Mathematics},
	year = {2013},
	volume = {71},
	pages = {201-213},
	number = {2},
	month = {},
	publisher = {}
}

@ARTICLE{moser,
	author = { Moser, R.},
	title = {Boundary Vortices
for Thin Ferromagnetic Films},
	journal = {Archive for Rational Mechanics and Analysis},
	year = {2004},
	volume = {174},
	pages = {267–300},
	number = {},
	month = {},
	publisher = {}
}

@ARTICLE{moser2,
	author = { Moser, R.},
	title = {Ginzburg-{L}andau vortices for thin ferromagnetic films},
	journal = {Applied Mathematics Research eXpress},
	year = {2003},
	volume = {},
	pages = {1-32},
	number = {1},
	month = {},
	publisher = {}
}

@ARTICLE{moser3,
	author = { Moser, R.},
	title = {Moving boundary vortices for a thin-film limit in micromagnetics},
	journal = {Communications on Pure and Applied Mathematics
},
	year = {2005},
	volume = {58},
	pages = {701-721},
	number = {5},
	month = {},
	publisher = {}
}

@ARTICLE{otto,
	author = {Ignat, R. and Otto, F.},
	title = {A compactness result for {L}andau state in thin-film micromagnetics},
	journal = {Annales de l{\textquotesingle}Institut Henri Poincar{\'{e}} C, Analyse non lin{\'{e}}aire},
	year = {2011},
	volume = {28},
	pages = {247-282},
	number = {2},
	month = {},
	publisher = {European Mathematical Society}
}

@ARTICLE{simone,
	author = { DeSimone, A. and  Kohn, R.  and Müller, S. and  Otto, F.},
	title = {A reduced theory for thin-film
micromagnetics},
	journal = {Communications on Pure and Applied Mathematics},
	year = {2002},
	volume = {55},
	pages = {1408-1460},
	number = {},
	month = {},
	publisher = {}
}

@ARTICLE{francois,
	author = {L'Official, F.},
	title = {The micromagnetic energy
with {D}zyaloshinskii-{M}oriya interaction
in a thin-ﬁlm regime relevant for boundary vortices},
	journal = {Communications on Pure and Applied Analysis},
	year = {2023},
	volume = {22},
	pages = {},
	number = {},
	month = {},
	publisher = {}
}

@ARTICLE{ignat+Lofficial,
	author = {Ignat, R. and  L’Official, F.},
	title = {Renormalised energy between boundary vortices in thin-film micromagnetics with {D}zyaloshinskii-{M}oriya interaction},
	journal = {Nonlinear Analysis},
	year = {2025},
	volume = {250},
	pages = {},
	number = {},
	month = {},
	publisher = {}
}

@ARTICLE{appliedfield,
	author = {Kurzke, M. and  Melcher, C. and  Moser, R.},
	title = {Vortex Motion for the Landau-Lifshitz-Gilbert
Equation with Applied Magnetic Field},
	journal = {Griebel, M. (eds) Singular Phenomena and Scaling in Mathematical Models},
	year = {2014},
	volume = {},
	pages = {},
	number = {},
	month = {},
	publisher = {Springer, Cham}
}

@BOOK{evans,
	title = {Partial Differential Equations},
	publisher = {American Mathematical Soc.},
	author = {Evans, L.},
	year = {2010},
	edition = {},}

@ARTICLE{knupfer,
	author = {Ignat, R. and  Kn\"{u}pfer, H.},
	title = {Vortex energy and $360^{\circ}$–N\'{e}el wall in thin–film
micromagnetics},
	journal={Communications on Pure and Applied Mathematics},
        year={2010},
     volume={63},
	pages = {},
	number = {},
	month = {},
	publisher = {Springer International Publishing Switzerland}
}

@ARTICLE{fratta,
   title={Reduced energies for thin ferromagnetic films with perpendicular anisotropy},
author={
Di Fratta, G. and  Muratov, C. and Slastikov, V.},
   volume={34},
   number={10},
   journal={Mathematical Models and Methods in Applied Sciences},
   publisher={World Scientific Pub Co Pte Ltd}, 
   year={2024},
   month={},
pages={1861–1904} 
}

@ARTICLE{Hile,
   title={Gradient bounds for harmonic functions Lipschitz on the boundary},
author={Hile, G. and Stanoyevitch, A.},
   volume={73},
   number={},
   journal={Applicable Analysis},
   publisher={}, 
   year={1999},
   month={},
pages={101–113} 
}

@BOOK{Ryzhik,
	title = {Table of integrals, series, and products. Edited by Alan Jeffrey and Daniel Zwillinger},
	publisher = {Oxford: Academic Press},
	author = {Gradshteyn, I. S. and Ryzhik, I. M},
	year = {2007},
	edition = {7},}
\end{document}